\documentclass[12pt]{article}

\usepackage{amssymb}
\usepackage{amsmath}
\usepackage{amsbsy}
\usepackage{amscd}
\usepackage{amsfonts}
\usepackage{amsthm}
\usepackage{mathrsfs}
\usepackage{verbatim}
\usepackage{hyperref}
\usepackage{fullpage}
\usepackage{mathdots}
\usepackage{graphicx,subfigure}
\usepackage[english]{babel}
\usepackage[utf8]{inputenc}
\usepackage{tikz}
\usepackage[normalem]{ulem}

\usetikzlibrary{backgrounds,fit, matrix}
\usetikzlibrary{positioning}
\usetikzlibrary{calc,through,chains}
\usetikzlibrary{arrows,shapes,snakes,automata, petri}
\usepackage{authblk}
\usepackage{mathtools}

\theoremstyle{definition}
\newtheorem{Theorem}{Theorem}[section]
\newtheorem{Corollary}[Theorem]{Corollary}
\newtheorem{Lemma}[Theorem]{Lemma}
\newtheorem{Proposition}[Theorem]{Proposition}
\newtheorem{Example}[Theorem]{Example}

\newtheorem{Remark}[Theorem]{Remark}
\newtheorem{Definition}[Theorem]{Definition}
\newtheorem{Notation}[Theorem]{Notation}

\newcommand{\C}{{\mathbb C}}

\newcommand{\Z}{{\mathbb Z}}

\newcommand{\N}{{\mathbb N}}

\newcommand{\mb}[1]{\mathbb{ #1}}

\newcommand{\mc}[1]{\mathcal{#1}}

\newcommand{\mbf}[1]{\mathbf{ #1}}

\newcommand{\red}[1]{{\color{red} #1}}

\newcommand{\mtt}[1]{\mathtt{#1}}

\newcommand\bfU{\mathbf{U}}

\begin{document}

\title{Irreducible Unipotent Numerical Monoids}
\author[1]{Mahir Bilen Can}
\author[2]{Naufil Sakran}

\affil[1]{{\small Tulane University, New Orleans LA; \href{mailto:mahirbilencan@gmail.com}{mahirbilencan@gmail.com}}}
\affil[2]{{\small Tulane University, New Orleans LA; \href{mailto:nsakran@tulane.edu}{nsakran@tulane.edu}}}
	
\normalsize
	
\date{\today}
\maketitle
	
\begin{abstract}
In our earlier article~\cite{CanSakran} we initiated a study of the complement-finite submonoids of the group of integer points of a unipotent linear algebraic group. In the present article, we continue to develop tools and techniques for analyzing our monoids. In particular, we initiate a theory of ideals for unipotent numerical monoids. 
\medskip
		
\medskip
		
\end{abstract}

\section{Introduction}
This article builds on the observation that the theory of generalized numerical monoids is part of a broader family of finitely generated monoids that naturally emerge within algebraic groups. 
Our primary objective here is to determine several structural invariants of these more general monoids, while shedding new light on generalized numerical semigroups, which form a fast developing area of research in semigroup theory. 
We name here only a few of the most relevant articles since the actual list is very long:~\cite{DiPasqualeGillespieBryanPeterson, BTT2022, SinghaiLin2022, CistoTenorio2021,CDGS, CFU2019, CFPU2019,FPU2016}. 
Here, by a generalized numerical monoid, we mean a co-finite submonoid of $\N^n$ for some positive integer $n$. 
Aside from their direct relation to commutative algebra~\cite{DGMV,BGS2023}, such commutative monoids arise in coding theory as multipoint Weierstrass semigroups~\cite{Homma1996, Kim1994, Matthews2001, CarvalhoKato, TT2019, BeelenTutas}. 
The setup of our theory, which we will maintain throughout this article is as follows.

Let $\mathbf{U}(n,\C)$ denote the unipotent group consisting of all upper triangular unipotent matrices of size $n$ with entries from $\C$.
Let $G$ be a closed subgroup of $\mathbf{U}(n,\mathbb{C})$. 
The group of integer points of $G$ is denoted by $G_\Z$. 
We set 
\begin{align*}
G(\N):=\{A\in G_{\Z} : \text{all entries of $A$ are from $\N$}\}.
\end{align*} 
Then $G(\N)$ is a submonoid of $G_\Z$.
We call a finitely generated submonoid $\mc{S}\subset G(\N)$, where the set $G(\N) \setminus \mc{S}$ is finite, a \red{\em unipotent numerical monoid in $G$}.
For example, if $G$ is given by the unipotent algebraic group 
\begin{align*}
\left\{
\begin{bmatrix}
1 & a_1 & a_2 &  \dots & a_{n-1}  \\
0 & 1 & 0 & \dots & 0 \\
0 & 0 & 1 & \dots & 0 \\
\vdots & \vdots & \vdots & \ddots & \vdots\\
0 & 0 & 0 & \dots & 1
\end{bmatrix}
:\ \{a_1,\dots, a_{n-1}\} \subset \C\right\} \qquad (n\geq 2),
\end{align*}
which we denote by $\mathbf{P}(n,\C)$, then we recover all generalized numerical monoids of~\cite{FPU2016} as unipotent numerical monoids.
If we employ semialgebraic topology and consider semialgebraic groups, then we recover the theory of affine $\mathcal{C}$-semigroups, introduced in~\cite{GMV2018}, as part of our theory as well.

A tool that we introduce in our paper is a notion of a semigroup of relative ideals.
Here, a \red{\em relative right ideal} of a unipotent numerical monoid $\mc{S}$ is a subset $\mc{I}\subseteq G(\N)$ such that 
\begin{align*}
\mc{I}\mc{S}\subseteq \mc{I}.
\end{align*}
A \red{\em relative left ideal} is defined similarly. 
It is a subset $\mc{I}\subseteq G(\N)$ such that $\mc{SI}\subseteq \mc{I}$.
Finally, we call a subset $\mc{I}\subseteq G(\N)$ that is both a left and right relative ideal a \red{\em two-sided relative ideal} of $\mc{S}$.
For brevity, we will subsequently refer to a two-sided relative ideal simply as a `\red{\em relative ideal}.'
Before proceeding to articulate our results, we will state our rationale concerning the definitions of relative right (or relative left) ideals. 
We could adopt a more generalized approach, considering subsets $\mc{I} \subset G_\Z$ where $\mc{IS}\subseteq \mc{I}$ holds, along with an additional condition, such as requiring the existence of an element $A\in \mc{S}$ such that $A\mc{I} \subset \mc{S}$. 
However, using this broader definition complicates our presentation, potentially obscuring the focus of our article. 
Instead of pursuing this cumbersome route, which is unnecessary for our purposes, we have opted to work specifically with the subsets of $G(\N)$.

The set of all relative ideals of $\mathcal{S}$, denoted by $Rel_T(\mathcal{S})$, forms a monoid under the product $(\mathcal{I},\mathcal{J})\mapsto \mathcal{IJ}:=\{AB \mid A\in \mathcal{I}\text{ and } B\in \mathcal{J}\}$, where $\{\mathcal{I},\mathcal{J}\}\subset Rel_T(\mathcal{S})$. 
The neutral element of $Rel_T(\mathcal{S})$ is given by $\mathcal{S}$. 
In the sequel, we will prove that $Rel_T(\mathcal{S})$ is closed under finite intersections and finite unions.
Furthermore, our monoid has two important submonoids, denoted $Rel_T(\mathcal{S},\supseteq)$ and $Rel_T(\mathcal{S},\subseteq)$, which are sublattices. 
The former submonoid consists of relative ideals of $\mathcal{S}$ contained in $\mathcal{S}$. 
The latter submonoid, which is finite, consists of relative ideals that contain $\mathcal{S}$. 
It is worth mentioning that while we have both left and right-sided versions of these objects, our discussion in the introduction is limited to the two-sided version for the sake of simplicity in our presentation. 
Our first main result is the following assertion.

\begin{Theorem}\label{intro:T1}
An element $\mc{T}$ of $Rel_T(\mc{S},\subseteq)$ is an idempotent if and only if $\mc{T}$ is a unipotent numerical monoid.
Furthermore, if $\mc{S}$ is commutative, then the idempotents of $Rel_T(\mc{S},\subseteq)$ form a band. 
\end{Theorem}

By fixing a unipotent numerical monoid $\mc{S}$, we actually fix several interesting partial orders on the group of integer points of $G$. 
For $A$ and $B$ from $G_\Z$, we have 
\begin{enumerate}
\item $A \leq_{\mc{S},r} B \iff BA^{-1}\in \mc{S}$, 
\item $A \leq_{\mc{S},l} B \iff A^{-1}B\in \mc{S}$,
\item $A \leq_{\mc{S},t} B \iff \{A^{-1}B, BA^{-1}\}\subset \mc{S}$.
\end{enumerate}
Let us call these partial orders collectively the \red{\em $\mc{S}$-orders} on $G_\Z$.

\begin{Theorem}\label{intro:T4}
If $\leq$ is an $\mc{S}$-order on $G_\Z$, then the pair $(G(\N),\leq)$ is a locally finite poset.
\end{Theorem}
	
Just as the torsion monoid of a unipotent numerical monoid $\mc{S}$ has a strong order theoretic structure, the semigroups of ideals of $\mc{S}$ also have a strong order theoretic property. 
We observe that, for a subset $\mc{I}\subseteq \mc{S}$, the complement  
$\mc{S}\setminus \mc{I}$ is a lower order ideal of the poset $(\mc{S},\leq_{\mc{S},l})$ if and only if $\mc{I}$ is a right ideal of $\mc{S}$.  
Similarly, $\mc{S}\backslash\mc{I}$ is a lower order ideal of the poset $(\mc{S},\leq_{\mc{S},r})$ if and only if $\mc{I}$ is a left ideal of $\mc{S}$. 
\medskip

Two fundamental invariants of a numerical semigroup $S$ are its Frobenius number and the set of pseudo-Frobenius elements. 
For unipotent numerical monoids, there are various sets that serve similar purposes. 
However, the Frobenius number is replaced by the Frobenius set. 
We proceed to introduce these important gadgets.

Let $\mc{S}$ be a unipotent numerical monoid.
Let $\mc{I}$ be a subset of $\mc{S}$. 
The \red{\em left Frobenius set of $\mc{I}$} is defined by 
$\mtt{F}_l(\mc{I}):= \{ A \in G(\N) \mid A \notin \mc{S} \text{ and } A(G(\N)^*)\subseteq \mc{I}\}$.
The \red{\em left pseudo-Frobenius set of $\mc{I}$} is defined by 
$\mtt{PF}_l(\mc{I}):= \{ A \in G(\N) \mid A \notin \mc{S} \text{ and } AS^*\subseteq \mc{I}\}$.
The right sided as well as the two-sided versions of these sets are denoted by $\mtt{F}_r(\mc{I}),\mtt{F}_t(\mc{I}), \mtt{PF}_r(\mtt{I})$, and $\mtt{PF}_t(\mc{I})$. 
The definitions of these sets are similar to those of $\mtt{F}_l(\mc{I})$ and $\mtt{PF}_l(\mc{I})$. 
Notice that the inclusions $\mtt{F}_l(\mc{I}) \subseteq \mtt{PF}_l(\mc{I})$, $\mtt{F}_r(\mc{I})\subseteq \mtt{PF}_r(\mc{I})$, and $\mtt{F}_t(\mc{I})\subseteq \mtt{PF}_t(\mc{I})$ always hold true. However, often these inclusions are strict inclusions. 

We call a unipotent numerical monoid $\mc{S}$ \red{\em irreducible} if there are no two unipotent numerical monoids $\mc{T}_1$ and $\mc{T}_2$, properly containing $\mc{S}$ such that $\mc{T}_1\cap \mc{T}_2 = \mc{S}$. 	
The following theorems indicate that both Frobenius and the pseudo-Frobenius sets are critical notions for investigating the irreducibility.
	
\begin{Theorem}\label{intro:T5}
Let $\mc{S}$ be a unipotent numerical monoid.
If one of the following conditions hold true, then $\mc{S}$ is irreducible: 
\begin{enumerate}
\item[(1)] If $|\mtt{F}_l(\mc{S})|=1$ and for every $A\in G(\N)\setminus \mc{S}$ we have $\mtt{F}_l(\mc{S}) \cap A \mc{S} \neq \emptyset$; 
\item[(2)] If $|\mtt{F}_r(\mc{S})|=1$ and for every $A\in G(\N)\setminus \mc{S}$ we have $\mtt{F}_r(\mc{S}) \cap \mc{S}A \neq \emptyset$;
\item[(3)] If $|\mtt{F}_t(\mc{S})|=1$ and for every $A\in G(\N)\setminus \mc{S}$ we have $\mtt{F}_t(\mc{S}) \cap \mc{S}A \cap A\mc{S} \neq \emptyset$.
\end{enumerate}
\end{Theorem}

In Theorem~\ref{intro:T5}, the assumptions $|\mtt{F}_l (\mc{S})|=|\mtt{F}_r (\mc{S})|=|\mtt{F}_t (\mc{S})|=1$ are crucial as the following counter-example shows.

\begin{Example}
In Figure~\ref{F:necessary}, we depicted a commutative unipotent numerical monoid $\mc{S}$ in $\mathbf{P}(3,\N)\cong \N \times \N$. 
The shaded boxes indicate the lattice points that are contained in $\mc{S}$. 
Hence, $\mc{S}$ is given by 
\begin{align*}
\mc{S} = \mathbf{P}(3,\N) \setminus \{ (0,1),(0,3),(3,0),(1,0) \}. 
\end{align*}
Since $\mc{S}$ is commutative, the left, the right, and the two-sided Frobenius sets of $\mc{S}$ are all equal. 
It is given by $\mtt{F}_t(\mc{S}) = \{ (0,3), (3,0) \}$. 
It is easy to check that for every $A\in \mathbf{P}(3,\N) \setminus \mc{S}$ we have 
$\mtt{F}_t(\mc{S}) \cap A \mc{S} \neq \emptyset$.
However, $\mc{S}$ is not irreducible since it is given by the intersection $\mc{S}_1 \cap \mc{S}_2$, where 
\begin{align*}
\mc{S}_1 := \mathbf{P}(3,\N) \setminus \{ (0,1),(0,3) \} \ \text{ and } \ 
\mc{S}_2 := \mathbf{P}(3,\N) \setminus \{ (1,0),(3,0) \}.
\end{align*}

\begin{figure}[htp]
\begin{center}
\scalebox{.6}{
\begin{tikzpicture}[scale=1.2]

\draw (0,0)  node[fill=olive!45!, minimum size=1cm,draw] {};
\foreach \i in {4,...,6} {\draw (\i,0)  node[fill=olive!45!, minimum size=1cm,draw] {};};
\foreach \i in {1,...,6} {\draw (\i,1)  node[fill=olive!45!, minimum size=1cm,draw] {};};
\foreach \i in {1,...,6} {\draw (\i,2)  node[fill=olive!45!, minimum size=1cm,draw] {};};
\foreach \i in {1,...,6} {\draw (\i,3)  node[fill=olive!45!, minimum size=1cm,draw] {};};
\foreach \i in {0,...,6} {\draw (\i,4)  node[fill=olive!45!, minimum size=1cm,draw] {};};
\foreach \i in {0,...,6} {\draw (\i,5)  node[fill=olive!45!, minimum size=1cm,draw] {};};
\foreach \i in {0,...,6} {\draw (\i,6)  node[fill=olive!45!, minimum size=1cm,draw] {};};

\draw (0,2)  node[fill=olive!45!, minimum size=1cm,draw] {}; 
\draw (2,0)  node[fill=olive!45!, minimum size=1cm,draw] {};

\node at (5,0) {$(5,0)$};
\node at (2,0) {$(2,0)$};
\node at (1,1) {$(1,1)$};
\node at (1,2) {$(1,2)$};
\node at (2,1) {$(2,1)$};
\node at (0,2) {$(0,2)$};
\node at (0,5) {$(0,5)$};
\node at (0,3) {\color{purple}$(0,3)$};
\node at (3,0) {\color{purple}$(3,0)$};

\foreach \i in {0,...,7} {\draw [dashed] (\i-.5,-.5) -- (\i-.5,5);};
\foreach \i in {0,...,7} {\draw [dashed] (-.5,\i-0.5) -- (5.5,\i-0.5);};

\end{tikzpicture}}
\end{center}
\caption{A commutative unipotent numerical monoid in $\mathbf{P}(3,\N)$.}
\label{F:necessary}
\end{figure}
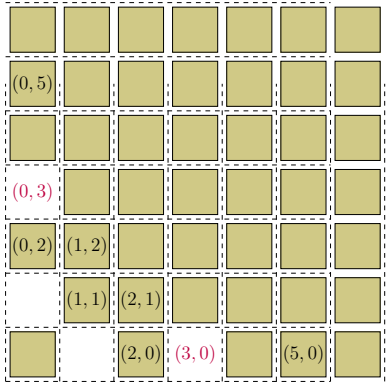
\end{Example}

We show in Example~\ref{E:conversefails} that the converse of Theorem~\ref{intro:T5} may fail to be true. 
In other words, there are irreducible unipotent numerical monoids such that $\mtt{F}_l(\mc{S}) \cap A \mc{S} = \emptyset$ for some $A\in G(\N)\setminus \mc{S}$.

\begin{Theorem}\label{intro:T6}
Let $\mc{S}$ be a unipotent numerical monoid. 
If $\mc{S}$ is irreducible, then we have $|\mtt{F}_t(\mc{S})|=1$.
\end{Theorem}

The converse of Theorem~\ref{intro:T6} also need not hold true. 
Simple counter examples exist already for the commutative cases. 
\medskip

The set of special gaps of a unipotent numerical monoid $\mc{S}$ is the subsets 
$\mtt{SG}(\mc{S})\subseteq \mtt{PF}_t(\mc{S})$ consisting of elements $A\in \mtt{PF}_t(\mc{S})$ such that $A^2 \in \mc{S}$. 
One of our main results is the following lattice theoretic characterization of the irreducible monoids. 

\begin{Theorem}\label{intro:T2}
Let $\mc{S}$ be a unipotent numerical monoid.	
Then $\mc{S}$ is irreducible if and only if the torsion monoid $Rel_T(\mc{S},\subseteq)$ has a unique minimal nontrivial idempotent and the two-sided Frobenius set of $\mc{S}$ equals the set of special gaps.
\end{Theorem}

A key concept in the study of numerical semigroups is the notion of a symmetric numerical semigroup. 
It has various equivalent descriptions, one of which is that for all $x\in \Z\setminus S$, the Frobenius number $|\mtt{F}(S)| - x$ is an element of $S$. 
Then it is known that $S$ is irreducible. 
In our setup, an adaptation of this definition is suggested by Theorem~\ref{intro:T5}.
We call an irreducible unipotent numerical monoid \red{\em symmetric} if for every $A\in G(\N)\setminus \mc{S}$ we have 
$$
\mtt{F}_t(\mc{S}) \cap A\mc{S} \neq \emptyset \quad\text{or}\quad   \mtt{F}_t(\mc{S})\cap \mc{S}A \neq \emptyset.
$$
Along similar lines, we call an irreducible unipotent numerical monoid \red{\em pseudo-symmetric} if the following two conditions are satisfied:
\begin{enumerate}
\item there exists an element $B\in G(\N)\setminus \mc{S}$ such that $B^2 \in \mtt{F}_t(\mc{S})$, and 
\item for every $A\in G(\N)\setminus (\mc{S} \sqcup \{B\})$ we have 
$\mtt{F}_t(\mc{S}) \cap A\mc{S} \neq \emptyset$ or $ \mtt{F}_t(\mc{S})\cap \mc{S}A \neq \emptyset$.
\end{enumerate}
An interesting property of the symmetric unipotent numerical monoids is that their Frobenius sets overlap with their pseudo-Frobenius sets (Proposition~\ref{P:PF=F}).
We subsequently show by an example that this property need not hold for every pseudo-symmetric unipotent numerical monoid.

The definitions of symmetric and pseudo-symmetric unipotent numerical monoids allow us to partition the set of all irreducible unipotent numerical monoids.
In this regard, we have the following result.  
\begin{Theorem}\label{intro:T7}
Let $\mc{S}$ be a unipotent numerical monoid. 
If $\mc{S}$ is irreducible, then $\mc{S}$ is either symmetric or pseudo-symmetric. 
\end{Theorem}
In fact, to establish this theorem, we prove an even stronger result (Theorem~\ref{intro:T7'}) within the main body of the text.
\medskip

Let $\mc{X}$ be a subset of $G(\N)$.
Let $\leqslant$ be a partial order on $G(\N)$ such that $\mathbf{1}_n$ is the unique minimal element with respect to $\leqslant$. 
For $A\in G(\N)$ and the poset $(\mc{X},\leqslant)$, we will use the following notation:
\begin{align*}
\mtt{n}(\mc{X},A, \leqslant) &:= \{ B\in \mc{X} \mid \mathbf{1}_n \leqslant B  \leqslant A \}.\\
\mtt{g}(\mc{X},A, \leqslant) &:= \{ B\in G(\N)\setminus \mc{X} \mid \mathbf{1}_n \leqslant B  \leqslant A \}.
\end{align*}
Note that we always have $\mtt{n}(G(\N) \setminus \mc{X},A, \leqslant) = \mtt{g}(\mc{X},A, \leqslant)$.
Hereafter we use $\preceq$ to denote the partial order $\leq_{G(\N),t}$ on $G(\N)$. 
	
\begin{Theorem}\label{intro:T8}
Let $\mc{S}$ be a symmetric unipotent numerical monoid in $G(\N)$ such that $\mtt{F}_t(\mc{S}) = \{C\}$ for some $C\in G(\N)\setminus \mc{S}$. 
If the additional condition $\mtt{F}_t(\mc{S}) \cap A\mc{S} \cap \mc{S}A \neq \emptyset$ holds, then we have 
\begin{align*}
|\mtt{n}(\mc{S},C, \preceq) | = |\mtt{g}(\mc{S},C, \preceq)|= \mtt{g}(\mc{S}).
\end{align*}
\end{Theorem}
\medskip

In a recent publication~\cite{DGSR}, Delgado, Garc\'{\i}a-S\'{a}nchez, and Rosales highlighted several important unsolved problems in the theory of numerical semigroups. 
It would be fascinating to explore how these problems might naturally extend to our non-commutative framework. 
\medskip
	
We are now ready to describe the structure of our paper and mention some additional results.
In Section~\ref{S:Preliminaries} we setup our notation and recall a few facts that were established in our earlier work~\cite{CanSakran}. 
In Section~\ref{S:Apery} we discuss the minimal generating sets of unipotent numerical monoids.
We introduce the noncommutative analogs of the Ap\'ery sets. 
For example, if $A$ is an element of the complement $G(\N)\setminus \mc{S}$, then the associated \red{\em left Ap\'ery set} is defined by $\mtt{Ap}_l(\mc{S}, A) := \{ B\in \mc{S} \mid A^{-1}B\notin \mc{S}\}$. 
We prove in Theorem~\ref{T:Aperygenerates} that $\mtt{Ap}_l(\mc{S},A) \cup \{A\}$ is a generating set for $\mc{S}$. 
Also in this section we show that the unique minimal generating set of $\mc{S}$ is given by $\mc{S}^*\setminus (\mc{S}^* \mc{S}^*)$. This is our Theorem~\ref{T:uniquemgs}.
In Section~\ref{S:Relative} we initiate the theory of relative ideals for unipotent numerical monoids. 
Theorem~\ref{intro:T1} is proven in this section. 
In Section~\ref{S:Finite}, 
we discuss the minimal generating sets of relative ideals.
We prove our Theorem~\ref{intro:T4}. 
In addition we show that a subset $\mc{I}\subseteq \mc{S}$ is a right ideal if and only if the complement $\mc{S}\setminus \mc{I}$ is a lower order ideal of the poset $(\mc{S},\leq_{\mc{S},l})$. We record this result as Theorem~\ref{T:lowerideal}.
As we indicated earlier, one of the most important ingredient of our work is a notion of a (right, left, or two-sided) Frobenius set. 
We introduce this concept in Section~\ref{S:Frobenius}, and we prove our Theorem~\ref{intro:T5}. 
Then in Section~\ref{S:Pseudo}, we introduce the corresponding pseudo-Frobenius sets. 
We prove Theorem~\ref{intro:T6} in that section.
Additionally, we discuss the relationships between the pseudo-Frobenius sets and the Ap\'ery sets, generalizing the well-known situation from the setting of numerical semigroups to the unipotent numerical monoids. 
The last result of Section~\ref{S:Pseudo} relates the irreducibility of unipotent numerical monoid to its torsion monoid and the set of special gaps. 
More precisely, we prove Theorem~\ref{intro:T2}.
Finally, in Section~\ref{S:Irreducibility}, we discuss the symmetric and pseudo-symmetric unipotent numerical monoids.
We prove our Theorems~\ref{intro:T7} and~\ref{intro:T8}.

\section{Notation and Preliminaries}\label{S:Preliminaries}

The set of positive integers (resp. nonnegative integers) is denoted by $\Z_+$ (resp. $\N$). 
For $n\in \Z_+$, the $n\times n$ identity matrix will be denoted by $\mathbf{1}_n$. 
Let $\mc{M}$ be a monoid.
Let $*$ denote the multiplication in $\mc{M}$.
Let $\mc{B} \subset \mc{M}$.  
We denote by $\langle \mc{B} \rangle$ the submonoid generated by $\mc{B}$ and the unit $e$ of $\mc{M}$, that is, 
\begin{align*}
\langle \mc{B} \rangle :=\{a_1*a_2*\cdots * a_k\,:\,\{a_1,\dots, a_k\} \subset \mc{B},\ k\in \N\} \cup \{e\}.
\end{align*}
We call $\mc{M}$ a \red{\em cancellative monoid} if for all $a, b, c \in \mc{M}$ the following implications hold true: 
\begin{itemize}
    \item Left cancellation law: if $ a* b = a*c $, then $ b = c $.
    \item Right cancellation law: if $ b*a = c*a $, then $ b = c $.
\end{itemize}
In other words, a cancellative monoid is a monoid in which the cancellation property holds for both left and right multiplication.
For example, if $\mc{M}$ is a submonoid of a group, then $\mc{M}$ is automatically a cancellative monoid. 
All unipotent numerical monoids that we consider in this article are cancellative monoids.

Let $n\in \Z_+$. 
Hereafter, $G$ will denote either the full unipotent group $\mbf{U}(n,\C)$ or a closed subgroup of $\mbf{U}(n,\C)$ such that $G(\N)$ is finitely generated. 
In addition we require that $G_\Z$ is the group completion of the monoid $G(\N)$. 
This is indeed the case for the group $\mathbf{U}(n,\Z)$ since both $\mathbf{U}(n,\N)$ and $\mathbf{U}(n,\Z)$
are generated by the same set of elements, which we will review below. 
\medskip

We observe that when $n=2$, the unipotent numerical monoids in $\mathbf{U}(n,\Z)$ coincide with the numerical semigroups in $\N$. 
Consequently, our results in this article can be regarded as extensions of some fundamental concepts of numerical semigroups into a noncommutative framework.
It is worth mentioning for motivational purposes that the theory of numerical semigroups is a classical subject in algebra, finding applications in algebraic geometry, coding theory, cryptography, number theory, and combinatorics. Comprehensive introductory material on numerical semigroups can be found in the monographs~\cite{ADG} and~\cite{RGS2}. The natural extensions of numerical semigroups are termed `generalized numerical semigroups.' 
These semigroups are finitely generated co-finite submonoids of $\N^{n-1}$.
In particular, they are commutative monoids. 
For a deeper understanding of finitely generated cancellative commutative monoids, readers are directed to the monograph~\cite{RosalesGarciaSanchez}. 
These resources also offer comprehensive bibliographies, serving as valuable gateways to the commutative semigroup literature.
\medskip
	
Let $\mc{S}$ be a unipotent numerical monoid in $G(\N)$. 
Many important features of $\mc{S}$ become visible once it is intersected with the \red{\em $k$-th fundamental unipotent monoid} defined by 
\begin{align*}
\bfU(n,\N)_{k}:= \{ (x_{ij})_{1\leq i,j\leq n}\in \mathbf{U}(n,\C):\ x_{ij}\in \N\ \text{ and }\ k\leq \max_{1\leq i,j\leq n} x_{ij}  \} \cup \{\mbf{1}_n\},
\end{align*}
where $k\in \Z_+$.
The {\em generating number of $\mc{S}$}, denoted $\mtt{r}(\mc{S})$, is the smallest positive integer $k\in \Z_+$ such that $(\bfU(n,\N)_{k}\cap G(\N)) \subseteq \mc{S}$.
The following sets are among the essential ingredients of unipotent numerical monoids:
\begin{itemize}
\item $\mathtt{Gaps}(\mc{S}):= G(\N) \setminus \mc{S}$ and $\mtt{g}(\mc{S}):= |\mathtt{Gaps}(\mc{S}) |$,
\item $\mathtt{N}(\mc{S}):= \{\mbf{1}_n \} \cup \mc{S} \setminus \mathbf{U}(n,\C)_{\mtt{r}(\mc{S})}$ and $\mtt{n}(\mc{S}):=|\mtt{N}(\mc{S})|$.
\end{itemize}
Let $\mathbf{d}_G$ denote the dimension of the ambient algebraic group $G$. 
While the concept of dimension has not been explicitly defined, for simplicity, we refer to $\mathbf{d}_G$ as the dimension of $\mathcal{S}$ as well.
Next, we define 
\begin{center}
\begin{tabular}{llll}
$\mtt{c}(\mc{S})$ &:=& $\mtt{r}(\mc{S})^{\mtt{d}_G}$, & \qquad called the {\em conductor of $\mc{S}$};\\
$\mtt{g}(\mc{S})$ &:=& $|G(\N) \setminus \mc{S}|$, & \qquad called the {\em genus of $\mc{S}$};\\
$\mtt{n}(\mc{S})$ &:=& $|\{\mbf{1}_n\}\cup \mc{S} \setminus \bfU(n,\N)_{\mtt{r}(\mc{S})}|$, & \qquad called the {\em sporadicity of $\mc{S}$}.
\end{tabular}
\end{center}

Let us identify for $\bfU(n,\N)$ ($n\geq 2$) the unique minimal system of generators. 
Let $\{i,j\} \subset \{1,\dots, n\}$. 
The {\em $(i,j)$-th elementary matrix of size $n$} is the matrix
\begin{align*}
E_{i,j}(n):=(e_{r,s})_{1\leq r,s \leq n},\quad \text{ where }\quad 
e_{r,s} :=
\begin{cases}
1 & \text{ if $r=s$};\\
1 & \text{ if $r=i,\ s=j$};\\
0 & \text{ otherwise.}
\end{cases}
\end{align*}
When the integer $n$ is evident from the context, we will exclude it from the notation of elementary matrices.
The following simple observation was recorded in~\cite{CanSakran}.
\begin{Proposition}\label{P:mgforUn}
The set $\{ E_{i,j} :\ 1\leq i < j \leq n\}$ is the unique minimal generating set for $\bfU(n,\N)$.
\end{Proposition}

The monoid $\mathbf{U}(n,\N)$ is filtered by the fundamental monoids. 
Indeed, for $k\in \N$, let us denote by $M_k$ the $(k+1)$-th fundamental monoid of $\mathbf{U}(n,\N)$. 
Then we have the containments, 
\begin{align}\label{A:descendingchain}
\mathbf{U}(n,\N) = M_0 \supset M_1 \supset M_2 \supset \cdots.
\end{align}
Recall from the introduction that $\mathbf{P}(n,\C)$ is the closed subgroup of $\mathbf{U}(n,\C)$ whose only nonzero terms are on its first row and the main diagonal. 
The {\em $k$-th fundamental monoid of $\mathbf{P}(n,\N)$} is defined by 
\begin{align*}
\mathbf{P}(n,\N)_k :=\bfU(n,\N)_{k}\cap \mathbf{P}(n,\N).
\end{align*}
Similarly to the monoid $\mathbf{U}(n,\N)$, the commutative monoid $\mathbf{P}(n,\N)$ is filtered by the fundamental monoids as well. 
For $k\in \N$, let $M_k'$ denote $\mathbf{P}(n,\N)_{k+1}$.
Then we have 
\begin{align}\label{A:descendingchain}
\mathbf{P}(n,\N) = M_0' \supset M_1' \supset M_2' \supset \cdots.
\end{align}

A partial order on $\mathbf{U}(n,\Z)$ that we make use of is the \red{\em entrywise order}, which is defined as follows:
\begin{align*}
A\leq_e B \iff a_{ij}\leq b_{ij}\quad \text{ for every $1\leq i,j\leq n$},
\end{align*}
where $A=(a_{ij})_{1\leq i,j\leq n}$ and $B=(b_{ij})_{1\leq i,j\leq n}$ are from $G_\Z$. 
\medskip

One of our notational conventions regarding partial orders is as follows: 
if $(P, \leq)$ is a partially ordered set, then $\max_\leq (P)$ denotes the set of maximal elements of $P$, not the maximum element of $P$.

\medskip

We close this section by introducing one more notation. 
For $\mc{I}\subseteq G(\N)$, we denote by $\mc{I}^*$ the set difference 
\begin{align*}
\mc{I}^*:= \mc{I} \setminus \{\mathbf{1}_n\}.
\end{align*}

\section{Ap\'ery Sets of Unipotent Numerical Monoids}\label{S:Apery}

A numerical semigroup is a complement-finite subsemigroup of $\N$.  
Let $S$ be a numerical semigroup. 
The subsemigroup of nonzero elements of $S$ is typically denoted by $S^*$.
Let $n\in S^*$. 
The Ap\'ery set, denoted by $\mathop{Ap}(n,S)$, is the set of elements $x\in S$ such that $x$ can not be decomposed as $x=n+s$ for any $s\in S$. 
Equivalently, we have 
\begin{align}
\mathop{Ap}(n,S):= \{ x\in S:\ x-n\notin S\}.
\end{align}
These sets provide essential tools for the study of a numerical semigroup. 
Our goal in this section is to generalize the Ap\'ery sets to the setting of the unipotent numerical monoids.

\begin{Definition}
Let $\mc{S}$ be a unipotent numerical monoid. 
Let $A\in \mc{S}^*$. 
The \red{\em left Ap\'ery set of $(\mc{S},A)$} is defined by 
\begin{align*}
\mtt{Ap}_l(\mc{S}, A) := \{ B\in \mc{S} \mid A^{-1}B\notin \mc{S} \}.
\end{align*}
The \red{\em right Ap\'ery set of $(\mc{S},A)$} is defined by 
\begin{align*}
\mtt{Ap}_r(\mc{S}, A) := \{ B\in \mc{S} \mid BA^{-1}\notin \mc{S} \}.
\end{align*}
The \red{\em two-sided Ap\'ery set of $(\mc{S},A)$} is defined by 
\begin{align*}
\mtt{Ap}_{t}(\mc{S}, A) := \{ B\in \mc{S} \mid \{ A^{-1}B, BA^{-1}\}\cap \mc{S} =\emptyset \}.
\end{align*}
Note that we always have 
\begin{align*}
\mtt{Ap}_{t}(\mc{S}, A)= \mtt{Ap}_{l}(\mc{S}, A)\cap \mtt{Ap}_{r}(\mc{S}, A).
\end{align*}
\end{Definition}
	
\begin{Theorem}\label{T:Aperygenerates}
Let $\mc{S}$ be a unipotent numerical monoid. 
Let $A\in \mc{S}^*$. 
Then we have 
\begin{enumerate}
\item $L_l:=\mtt{Ap}_l(\mc{S},A) \cup \{A\}$ is a generating set for $\mc{S}$. 
\item $L_r:=\mtt{Ap}_r(\mc{S},A) \cup \{A\}$ is a generating set for $\mc{S}$.
\end{enumerate}
\end{Theorem}

\begin{proof}
The proof of (2) is very similar to the proof of (1). 
We will prove (1) only.

Since $A^{-1}$ has a negative entry, we see that $\mbf{1}_n\in \mtt{Ap}_l(\mc{S},A)$. 
Now let $B\in\mc{S}^*$. If $B\in L_l$, then we have nothing to prove. 
We proceed with the assumption that $B\in\mathcal{S}\backslash L_l$. 
Since $B$ is not an element of $\mtt{Ap}_l(\mc{S},A)$, we have $A^{-1}B\in \mc{S}^*$. 
If $A^{-1}B$ is an element of $\mtt{Ap}_l(\mc{S},A)$, then by writing $B$ in the form $B= A(A^{-1}B)$ we see that $B \in \langle L_l\rangle$. 
Let us assume that $A^{-1}B$ is not an element of $\mtt{Ap}_l(\mc{S},A)$. 
Then we have $A^{-2}B\in\mc{S}^*$. 
If $A^{-2}B\in \mtt{Ap}_l(\mc{S},A)$, then the proof is finished since we can write $B=A^2 (A^{-2}B)$.
We now assume that $A^{-2}B\notin \mtt{Ap}_l(\mc{S},A)$.
Then we have $A^{-3}B\in\mc{S}^*$. 
Continuing inductively, we notice that this process must stop after a finite number of steps since at least one entry of $A^{-r}B$, $r=0,1,2,\dots$ continues to decrease indefinitely. 
Hence, we conclude that for some $k\in\mb{N}$ the product $A^{-k}B$ is an element of $\mtt{Ap}_l(\mc{S},A)$. 
But this means that $B=A^k\,(A^{-k}B)$, implying $B\in\langle L_l\rangle$.
This finishes the proof of our theorem.
\end{proof}

After proving the existence of a minimal generating set for $\mc{S}$ we will show that the two-sided Ap\'ery set is also a generating set.

In our earlier work, we proved that the unipotent numerical monoids have a unique minimal generating set. 
Now we will write down a practical description of such a generating set.

\begin{Theorem}\label{T:uniquemgs}
Let $\mc{S}$ be a unipotent numerical monoid in $G(\N)$, where $G$ is a closed subgroup of $\mathbf{U}(n,\C)$. 
Then the unique minimal generating set of $\mc{S}$ is given by $\mc{S}^*\setminus (\mc{S}^* \mc{S}^*)$. 
\end{Theorem}

\begin{proof}
Let $k\geq 2$, and define 
\begin{align*}
\mc{Q}_k:= \{ (x_{ij})_{1\leq i,j\leq n}  \in \mbf{U}(n,\N)  :\ k\leq \max_{1\leq i,j\leq n} x_{ij} < 2k \}.
\end{align*}
We know from~\cite[Lemma 2.9]{CanSakran} that $\mathcal{Q}_k\sqcup\{\textrm{sporadic elements}\}$, where $k$ is the generating number of $\mc{S}$, is a generating set for $\mc{S}$ despite the fact that it might be non-minimal. 
Let $\mathcal{A}=\mc{S}^*\setminus (\mc{S}^*\mc{S}^*)$ and $\mathcal{B}$ be another minimal generating set. 
Let $B\in \mathcal{B}\setminus \mathcal{A}$. 
Then $B\in \mc{S}^*\mc{S}^*$.
Therefore, we have two elements $A,C\in \mc{S}^*$ such that $B=AC$. 
Since $A$ and $C$ are contained in $\langle\mathcal{B}\rangle$, we see that $\mathcal{B}\setminus\{B\}$ is a generating set, contradicting the minimality of $\mathcal{B}$. 
In conclusion, we have $\mathcal{B}\setminus \mathcal{A}=\emptyset$. 
In other words, $\mathcal{B}\subseteq \mathcal{A}$ holds. 
\medskip

For the reverse inclusion, let $A\in \mathcal{A}$. 
Then we write $A$ as a product of the elements of $\mathcal{B}$ as follows: $A=\prod_{i}B_i^{n_i}$ where $B_i\in \mathcal{B}$ for every $i$. 
Since $A\in \mathcal{A}$, we have a unique index $i$ such that $n_i=1$, and $n_j=0$ for every $j\neq i$. 
Hence, $A=B_i\in \mathcal{B}$ implying that $\mathcal{A}\subseteq \mathcal{B}$. 
This finishes the proof of our theorem.
\end{proof}

\begin{Corollary}
Let $\mc{S}$ be a unipotent numerical monoid. 
Let $A\in \mc{S}$. 
Then $L_{t}:= \{A\}\cup \mtt{Ap}_{t}(\mc{S},A)$ is a generating set for $\mc{S}$.
\end{Corollary}
\begin{proof}
We already know that $L_l:= \mtt{Ap}_{t}(\mc{S},A) \cup \{A\}$ and $L_r:=\mtt{Ap}_{t}(\mc{S},A) \{A\}$ are generating sets for $\mc{S}$. 
Let $\mc{E}=\mc{S}^*\setminus (\mc{S}^*\mc{S}^*)$ denote the unique minimal generating set of $\mc{S}$. 
Then $\mc{E}$ is a subset of both $L_l$ and $L_r$. 
Otherwise, for $A\in \mc{E} \setminus L_l$, the element $A$ does not belong to the subset $\langle L_l \rangle$, which is absurd. 
But it is evident from the definitions that $L_{t}= L_l \cap L_r$.
Hence, $L_{t}$ contains $\mc{E}$. 
This finishes the proof of our assertion.
\end{proof}

\section{Semigroups of Relative Ideals}\label{S:Relative}

Throughout this section, we assume that $\mc{S}$ is a unipotent numerical monoid of a unipotent algebraic subgroup $G\subset \mathbf{U}(n,\C)$.
The cases we have in mind are $G= \mathbf{U}(n,\C)$ or $G=\mathbf{P}(n,\C)$.
The following notational conventions will be used throughout the paper.

Let $\mc{I}$ and $\mc{J}$ be two subsets of $\mc{S}$.
The set $\{ XY\mid X\in \mc{I},\ Y\in \mc{J} \}$ will be denoted by $\mc{IJ}$. 
We now proceed to define the notion of a relative ideal.

Recall that a subset $\mc{I} \subseteq G(\N)$ is called a \red{\em relative right ideal} of $\mc{S}$ if the inclusion $\mc{I}\mc{S} \subseteq \mc{I}$ holds. 

\begin{Example}\label{E1:SandS^*}
Both $\mc{S}$ and its maximal subsemigroup $\mc{S}^*$ are automatically right relative ideals of $\mc{S}$. 
\end{Example}

We denote the set of all relative right ideals by $Rel_R(\mc{S})$. 
Now, if $T$ is a semigroup, an element $e\in T$ is said to be a \red{\em right identity element} of $T$ if $xe =x$ holds for all $x\in T$.

\begin{Proposition}\label{P:rightidentity}
The product $(\mc{I},\mc{J}) \mapsto \mc{IJ}$ endows $Rel_R(\mc{S})$ with a structure of a semigroup.
Furthermore, the relative right ideal $\mc{S}$ is a right identity element of $Rel_R(\mc{S})$.
\end{Proposition}

\begin{proof} 
First of all, if $(\mc{I},\mc{J}) \in Rel_R(\mc{S})\times Rel_R(\mc{S})$, then  
\begin{align*}
\mc{I}\mc{J} := \{ XY\mid X\in \mc{I},\ Y\in \mc{J} \}
\end{align*}
is a relative right ideal of $\mc{S}$. 
This follows from the following simple observation: 
\begin{align*}
(\mc{I}\mc{J}) \mc{S}= \mc{I} (\mc{J}\mc{S}) \subseteq \mc{I} \mc{J}.
\end{align*}
The associativity of the product $(\mc{I},\mc{J}) \mapsto \mc{IJ}$ follows from the associativity of the matrix multiplication. 
We already pointed out in Example~\ref{E1:SandS^*} that $\mc{S} \in Rel_R(\mc{S})$. 
Since for every relative right ideal $\mc{I}$ of $\mc{S}$ we have $\mc{I}\mc{S}\subseteq \mc{I}$, and since $\mc{S}$ contains the identity matrix $\mathbf{1}_n$, we see that $\mc{I}\mc{S} = \mc{I}$. 
Therefore, $\mc{S}$ is a right identity element of $Rel_R(\mc{S})$.  
This finishes the proof of our proposition.
\end{proof}

Recall that a \red{\em relative left ideal of $\mc{S}$} is a subset $\mc{I}\subseteq G(\N)$ such that $\mc{S} \mc{I} \subseteq \mc{I}$.
A subset $\mc{I}\subseteq \mc{S}$ is said to be a \red{\em relative (two-sided) ideal of $\mc{S}$} if it is both a relative right ideal and a relative left ideal of $\mc{S}$.
A \red{\em right ideal} (resp. \red{\em left ideal}, resp. \red{\em ideal}) of $\mc{S}$ is a relative right (resp. left, resp. two-sided) ideal $\mc{I}\subset G(\N)$ such that $\mc{I}\subseteq \mc{S}$. 
We call a right (resp. left, resp. two-sided) ideal of $\mc{S}$ \red{\em proper} if $\mc{I}\neq \mc{S}$.

\begin{Example}\label{E2:fundamental}
Let $\mc{S}$ be a unipotent numerical monoid in $\mathbf{U}(n,\N)$. 
Recall that $\mtt{r}(\mc{S})$ denotes the generating number of $\mc{S}$. 
Let $k$ be an integer such that $k\geq \mtt{r}(\mc{S})$. 
Then the $k$-th fundamental monoid $\mathbf{U}(n,\N)_k$ is an ideal of $\mc{S}$. 
\end{Example}

Similarly to how the semigroup $Rel_R(\mc{S})$ is constructed, we can analogously construct semigroups corresponding to relative left ideals and relative ideals. 
These semigroups are denoted as $Rel_L(\mc{S})$ and $Rel_{T}(\mc{S})$, respectively.

\begin{Theorem}\label{P:identity}
The set $Rel_T(\mc{S})$ is a monoid with the identity element $\mc{S}$. 
\end{Theorem}
\begin{proof}
The semigroup property of $Rel_T(\mc{S})$ is evident. 
It is also evident that the relative ideal $\mc{S}$ serves as both the right identity and the left identity element for the product $(\mc{I},\mc{J})\mapsto \mc{IJ}$, where $\{\mc{I},\mc{J}\}\subset Rel_T(\mc{S})$.
\end{proof}

We have two remarks in order.

\begin{Remark}
There are no nontrivial invertible elements in $Rel_T(\mc{S})$. 
Indeed, if $\mc{I}$ is an invertible element of $Rel_T(\mc{S})$ such that $\mc{I}\neq \mc{S}$, then there exists $\mc{J}\in Rel_T(\mc{S})$ such that $\mc{IJ}=\mc{S}$.
This means that both $\mc{I}$ and $\mc{J}$ contains $\mathbf{1}_n$. 
It follows that 1) $\mc{S}\subseteq \mc{I}$, and 2) $\mc{I}\subseteq \mc{IJ}=\mc{S}$. 
Then we find that $\mc{I}=\mc{S}$, which is absurd. 
\end{Remark}

\begin{Remark}
Let $\mc{S}$ be a commutative unipotent numerical monoid. Then we have 
\begin{align*}
Rel_L(S) = Rel_R(S)=Rel_T(S).
\end{align*} 
\end{Remark}

The elements $\mc{I}\in Rel_T(\mc{S})$ such that $\mathbf{1}_n\in \mc{I}$ form a finite submonoid of $Rel_T(\mc{S})$. 
Indeed, if $\mc{I}$ and $\mc{J}$ are two relative ideals containing $\mbf{1}_n$, then $\mathbf{1}_n$ is still an element of $\mc{IJ}$.
Furthermore, in this case we have $\mc{I}\cup \mc{J}\subseteq \mc{IJ}$.
Finally, if a relative ideal $\mc{I}$ contains $\mathbf{1}_n$, then it contains the whole unipotent numerical monoid $\mc{S}$. 
It follows that there are only finitely many relative ideals $\mc{I}$ of $\mc{S}$ such that $\mathbf{1}_n\in \mc{I}$.
\begin{Definition}
Let $\mc{S}$ be a unipotent numerical monoid. 
The \red{\em torsion monoid of $\mc{S}$} is the finite submonoid 
\begin{align*}
Rel_T(\mc{S},\subseteq) := \{ \mc{I}\in Rel_T(\mc{S}) \mid \mc{S}\subseteq\mc{I} \} \subsetneq Rel_T(\mc{S}).
\end{align*}
\end{Definition}

\begin{Lemma}\label{L:unmiffidempotent}
Let $\mc{I}$ be an element of $Rel_T(\mc{S},\subseteq)$. 
Then $\mc{I}$ is a unipotent numerical monoid if and only if $\mc{I}$ is an idempotent. 
\end{Lemma}
\begin{proof}
If $\mc{I}^2=\mc{I}$ holds, then $\mc{I}$ is closed under multiplication. 
Since $\mathbf{1}_n\in \mc{I}$ and $\mc{S}\subseteq \mc{I}$, $\mc{I}$ is a completement finite submonoid of $G(\N)$. 

Conversely, if $\mc{I}$ is a unipotent numerical monoid, then $\mc{I}$ is closed under products, that is, $\mc{I}^2\subseteq \mc{I}$. 
In particular, since $\mathbf{1}_n\in \mc{I}$, we see that $\mc{I}^2 = \mc{I}$. 
\end{proof}

\begin{Theorem}\label{T:closedunderunions}
Let $\mc{S}$ be a unipotent numerical monoid. Then the semigroups $Rel_R(\mc{S})$, $Rel_L(\mc{S})$, and $Rel_T(\mc{S})$ are closed under finite intersections and finite unions. 
\end{Theorem}
	
\begin{proof}
We will prove our theorem for $Rel_R(\mc{S})$ only since the proofs of the other two cases are similar. 

Let $\mc{I}$ and $\mc{J}$ be two relative right ideals of $\mc{S}$.
Then we have 
\begin{enumerate}
\item $(\mc{I}\cup \mc{J}) \mc{S} = (\mc{I}\mc{S}) \cup (\mc{J}\mc{S}) \subseteq (\mc{I}\cup  \mc{J})$, and 
\item $(\mc{I}\cap \mc{J}) \mc{S} \subseteq (\mc{I}\mc{S}) \cap (\mc{J}\mc{S}) \subseteq (\mc{I}\cap  \mc{J})$.
\end{enumerate}
Hence, the proof of our theorem follows. 
\end{proof}

\begin{Notation}
Hereafter, the set of all idempotents of a semigroup $X$ will be denoted by $E(X)$. 
The neutral element of $X$ will be referred to as the \red{\em trivial idempotent} of $X$. 
\end{Notation}

\begin{Corollary}\label{C:intersections}
Let $\mc{S}$ be a unipotent numerical monoid. 
Then $Rel_T(\mc{S},\subseteq)$ is closed under unions and intersections.
In particular, the set of idempotents $E(Rel_T(\mc{S},\subseteq))$ is a lattice. 
\end{Corollary}
\begin{proof}
The fact that the monoid $Rel_T(\mc{S},\subseteq)$ is closed under finite unions and finite intersections is clear from Theorem~\ref{T:closedunderunions}. 
Since $Rel_T(\mc{S},\subseteq)$ is finite and contains $G(\N)$ as the unique maximal element, to show that $E(Rel_T(\mc{S},\subseteq))$ is a lattice, it suffices to show that 
$E(Rel_T(\mc{S},\subseteq))$ is closed under intersections. 
Let $\mc{I}$ and $\mc{J}$ be two idempotents from $Rel_T(\mc{S},\subseteq)$.
To this end, we observe that
\begin{align*}
(\mc{I} \cap \mc{J})^2 \subseteq \mc{I}^2 \cap \mc{J}^2 = \mc{I}\cap \mc{J}.
\end{align*}
The reverse inclusion automatically holds since $\mathbf{1}_n \in \mc{I} \cap \mc{J}$.
Hence, the intersection $\mc{I}\cap \mc{J}$ is an idempotent as well. 
This finishes the proof of our corollary.
\end{proof}

In general, $Rel_T(\mc{S},\subseteq)$ and its set of idempotents $E(Rel_T(\mc{S},\subseteq))$ are not equal.

\begin{Example}
Let $\mc{S}$ denote the $k$-th fundamental monoid of $\mathbf{P}(n,\N)$. 
Then each element of $Rel_T(\mc{S},\subseteq)$ is of the form $\mc{S} \cup \bigcup_{A\in \mc{A}} A\mc{S}$, 
where $\mc{A}$ is a subset of the gap set of $\mc{S}$. 
We want to emphasize that not every element of $Rel_T(\mc{S},\subseteq)$ is a unipotent numerical monoid. 
Indeed, on the left side of Figure~\ref{F:2nd}, we have a relative ideal of $\mc{S}:=\mathbf{P}(3,\N)_2$, which is a unipotent numerical monoid itself. 
On the right side of the same figure, we have a relative ideal that is not a unipotent numerical monoid.
The points colored in blue are the minimal generators of the ideals. 

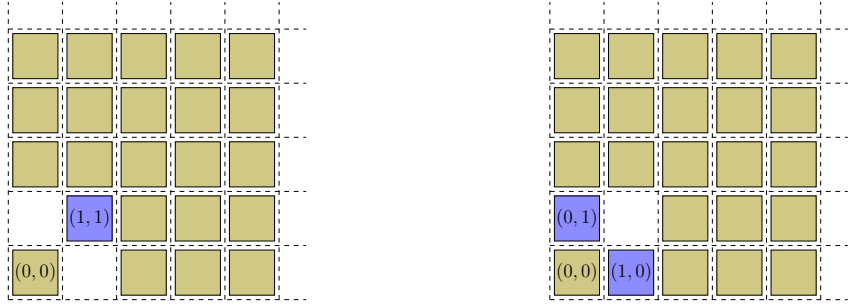
\begin{figure}[htp]
\begin{center}
\scalebox{.6}{
\begin{tikzpicture}[scale=1.2]
						
\begin{scope}[shift={(-5cm,0)}]
							
\foreach \i in {2,...,4} {\draw (\i,0)  node[fill=olive!45!, minimum size=1cm,draw] {};};
\foreach \i in {2,...,4} {\draw (\i,1)  node[fill=olive!45!, minimum size=1cm,draw] {};};
\foreach \i in {0,1,...,4} {\draw (\i,2)  node[fill=olive!45!, minimum size=1cm,draw] {};};
\foreach \i in {0,1,...,4} {\draw (\i,3)  node[fill=olive!45!, minimum size=1cm,draw] {};};
\foreach \i in {0,1,...,4} {\draw (\i,4)  node[fill=olive!45!, minimum size=1cm,draw] {};};
\foreach \i in {0,...,5} {\draw [dashed] (\i-.5,-.5) -- (\i-.5,5);};
\foreach \i in {0,...,5} {\draw [dashed] (-.5,\i-0.5) -- (5,\i-0.5);};

\draw (1,1)  node[fill=blue!45!, minimum size=1cm,draw] {};
\draw (0,0)  node[fill=olive!45!, minimum size=1cm,draw] {};
\node at (0,0) {$(0,0)$};
\node at (1,1) {$(1,1)$};
\end{scope}

\begin{scope}[shift={(5cm,0)}]

\foreach \i in {2,...,4} {\draw (\i,0)  node[fill=olive!45!, minimum size=1cm,draw] {};};
\foreach \i in {2,...,4} {\draw (\i,1)  node[fill=olive!45!, minimum size=1cm,draw] {};};
\foreach \i in {0,1,...,4} {\draw (\i,2)  node[fill=olive!45!, minimum size=1cm,draw] {};};
\foreach \i in {0,1,...,4} {\draw (\i,3)  node[fill=olive!45!, minimum size=1cm,draw] {};};
\foreach \i in {0,1,...,4} {\draw (\i,4)  node[fill=olive!45!, minimum size=1cm,draw] {};};
\foreach \i in {0,...,5} {\draw [dashed] (\i-.5,-.5) -- (\i-.5,5);};
\foreach \i in {0,...,5} {\draw [dashed] (-.5,\i-0.5) -- (5,\i-0.5);};

\draw (0,1)  node[fill=blue!45!, minimum size=1cm,draw] {};
\draw (1,0)  node[fill=blue!45!, minimum size=1cm,draw] {};
\draw (0,0)  node[fill=olive!45!, minimum size=1cm,draw] {};
\node at (0,0) {$(0,0)$};
\node at (0,1) {$(0,1)$};
\node at (1,0) {$(1,0)$};
\end{scope}

\end{tikzpicture}
}
\end{center}
\caption{Two relative ideals of $\mc{S}=\mathbf{P}(3,\N)_2$. The one on the left is a unipotent numerical monoid containing $\mc{S}$.}
\label{F:2nd}
\end{figure}

The idempotents of $Rel_T(\mc{S},\subseteq)$ are given by
\begin{align*}
\mc{T}_0 &:= \mc{S},\\ 
\mc{T}_1 &:= \mc{S} \cup (1,0)\mc{S},\\ 
\mc{T}_2 &:= \mc{S} \cup (1,1)\mc{S},\\ 
\mc{T}_3 &:= \mc{S} \cup (0,1)\mc{S},\\ 
\mc{T}_4 &:= \mc{S} \cup (1,0)\mc{S} \cup (1,1)\mc{S},\\ 
\mc{T}_5 &:= \mc{S} \cup (0,1)\mc{S} \cup (1,1)\mc{S},\\ 
\mc{T}_6 &:= \mathbf{P}(3,\N). 
\end{align*}

In Figure~\ref{F:Hasse}, we have the Hasse diagram of the lattice of idempotents $E(Rel_T(\mc{S},\subseteq))$.

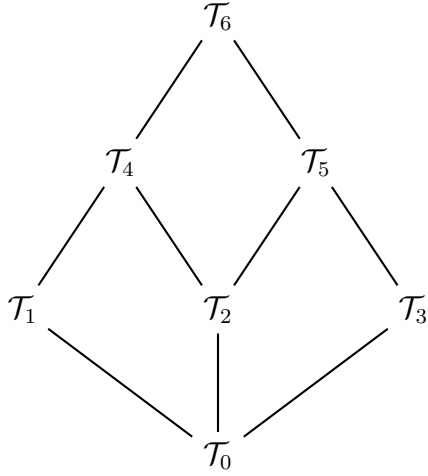
\begin{figure}[htp]
\begin{center}
\begin{tikzpicture}[scale=1.3]

\node (t0) at (0,-.5) {$\mc{T}_0$};
\node (t1) at (-2,1) {$\mc{T}_1$};
\node (t2) at (0,1) {$\mc{T}_2$};
\node (t3) at (2,1) {$\mc{T}_3$};
\node (t4) at (-1,2.5) {$\mc{T}_4$};
\node (t5) at (1,2.5) {$\mc{T}_5$};
\node (t6) at (0,4) {$\mc{T}_6$};

\draw[thick] (t0) to (t1);
\draw[thick] (t0) to (t2);
\draw[thick] (t0) to (t3);
\draw[thick] (t1) to (t4);
\draw[thick] (t2) to (t4);
\draw[thick] (t2) to (t5);
\draw[thick] (t3) to (t5);
\draw[thick] (t4) to (t6);
\draw[thick] (t5) to (t6);

\end{tikzpicture}
\end{center}
\caption{The Hasse diagram of the idempotent lattice $E(\mathbf{P}(3,\N)_2)$.}
\label{F:Hasse}
\end{figure}

\end{Example}

An idempotent semigroup is a semigroup whose elements are idempotents. 
Such semigroups are often referred to as \red{\em bands}.
\begin{Proposition}\label{P:itisaband}
Let $\mc{S}$ be a unipotent numerical monoid. 
If $\mc{S}$ is commutative, then $E(Rel_T(\mc{S},\subseteq))$ is a commutative band. 
\end{Proposition}
\begin{proof}
It suffices to show that if $\mc{I}$ and $\mc{J}$ are two idempotent ideals from $Rel_T(\mc{S},\subseteq)$,
then $(\mc{I}\mc{J})^2 = \mc{I}\mc{J}$ and $\mc{I}\mc{J} = \mc{J}\mc{I}$.
Since $\mc{S}$ is commutative, the elements of $\mc{I}$ and $\mc{J}$ commute with each other, showing that $\mc{I}\mc{J}=\mc{J}\mc{I}$.  
In particular, we have 
$$
(\mc{I}\mc{J})^2 = \mc{I}\mc{J}\mc{I}\mc{J}=\mc{I}\mc{I}\mc{J}\mc{J} = \mc{I} \mc{J}.
$$
This finishes the proof.
\end{proof}	
\begin{Remark}
In the proof of Proposition~\ref{P:itisaband}, replacing $Rel_T(\mc{S},\subseteq)$ with $Rel_T(\mc{S})$, we see that if $\mc{S}$ is a commutative unipotent numerical monoid, then the idempotents of $Rel_T(\mc{S})$ form a commutative band.
\end{Remark}
	
We are now ready to prove the first stated theorem from the introduction.
Let us recall its statement for the convenience of the reader. 
\medskip

An element $\mc{T}$ of $Rel_T(\mc{S},\subseteq)$ is an idempotent if and only if $\mc{T}$ is a unipotent numerical monoid.
Furthermore, if $\mc{S}$ is commutative, then the idempotents of $Rel_T(\mc{S},\subseteq)$ form a band. 

\begin{proof}[Proof of Theorem~\ref{intro:T1}]
The proof follows readily from Lemma~\ref{L:unmiffidempotent} and Proposition~\ref{P:itisaband}. 
\end{proof}

Recall that a unipotent numerical monoid $\mc{S}$ is called irreducible 
if it cannot be written as an intersection of the form $\mc{S}=\mathcal{T}_1 \cap \mathcal{T}_2$, where $\mathcal{T}_i$ ($i\in \{1,2\}$) is a unipotent numerical monoid properly containing $\mc{S}$. The torsion monoid of $\mc{S}$ captures this property in a lattice theoretic way. 
We will discuss this connection in a greater detail after we introduce the pseudo-Frobenius sets.

\section{Generating Sets and Order Ideals}\label{S:Finite}

In this section, we discuss the generating sets of relative ideals in relation with the partial orders that we introduced in the introduction section.

Let $\mc{I}$ be a relative right ideal of $\mc{S}$. 
A subset $\mathcal{E}\subseteq \mc{I}$ is called \red{\em a set of generators for $\mc{I}$} if 
every element $D\in \mc{I}$ can be written in the form $D=AB$ for some $A\in \mathcal{E}$ and $B\in \mc{S}$.
In other words, a subset $\mathcal{E}\subset \mc{I}$ is a set of generators for $\mc{I}$ if $\mc{I}$ is given by 
\begin{align*}
\mathcal{E}\mc{S} := \{ AB \mid A\in \mc{E},\ B\in \mc{S}\} = \mc{I}.
\end{align*}
A relative right ideal is said to be \red{\em finitely generated} if it has a finite set of generators.

\begin{Lemma}\label{L:mingens1}
The set $\mc{I}\setminus \mc{I}^*\mc{S}^*$ is the unique minimal generating set for $\mc{I}$.
\end{Lemma}

\begin{proof}
Let $\mc{F}$ denote $\mc{I}\setminus \mc{I}^*\mc{S}^*$.
The fact that $\mc{F}$ must be a part of any generating set for $\mc{I}$ does not require a proof. 
		
Let $X_0\in \mc{I}^*\mc{S}^*$. 
Then $X_0 = X_1A_1$ for some $X_1\in \mc{I}^*$ and $A_1\in \mc{S}^*$. 
Assume that $X_1\notin \mc{F}$.
Then, there exists $X_2 \in \mc{I}^*$ and $A_2\in \mc{S}^*$ such that $X_1 = X_2 A_2$.
Hence, we have $X_0 = X_2 A_2A_1$.
We proceed inductively to construct decompositions,
\begin{align*}
X_0 &= X_1A_1 \\ 
&= X_2 A_2 A_1 \\
&\vdots \\
&= X_m A_m A_{m-1} \cdots A_2 A_1\\
&\vdots
\end{align*}
where $A_i \in \mc{S}^*$ and $X_i\in \mc{I}^*$ for $i=1,2,\dots$.
But since all matrices involved in these decompositions are from $G(\N)$, in each decomposition, the entries of the matrices get smaller.
Hence, this process terminates after finitely many iterations. 
At the end, we arrive at a decomposition 
\begin{align*}
X_0 = X_r (A_r A_{r-1} \cdots A_2 A_1),
\end{align*}
where $X_r \in \mc{I}\setminus (\mc{I}^*\mc{S}^*)$ and $A_r A_{r-1} \cdots A_2 A_1\in \mc{S}^*$.
It follows that the set $\mc{I}\setminus (\mc{I}^*\mc{S}^*)$ is a minimal generating set for the relative right ideal $\mc{I}$.
\end{proof}

The generating set of a relative right ideal has a convenient interpretation via the partial orders that we described in the introduction. 
	
\begin{Lemma}\label{L:mingens2}
Let $\mc{I}$ be a relative right ideal.
Then the unique minimal generating set of $\mc{I}$ is given by the set of minimal elements of $\mc{I}^*$ with respect to 
$\leq_{\mc{S},l}$, denoted ${\rm Minimals}_{\leq_{\mc{S},l}} \mc{I}^*$. 
\end{Lemma}

\begin{proof}
Since every element $B\in \mc{I}^*$ is greater than or equal to some minimal element $A\in {\rm Minimals}_{\leq_{\mc{S},l}} \mc{I}^*$, we have $B\in A\mc{S}$. 
But this means that ${\rm Minimals}_{\leq_{\mc{S},l}} \mc{I}^*$ generates $\mc{I}$. 
Hence, the unique minimal generating set $\mc{I}\setminus \mc{I}^*\mc{S}^*$ is contained in ${\rm Minimals}_{\leq_{\mc{S},l}} \mc{I}^*$.
Conversely, suppose $A$ is an element of the minimum under $\leq_{\mc{S}, l}$ of $\mc{I}^*$, then $A$ cannot be contained in $\mc{I}^* \mc{S}^*$. 
Otherwise, this would contradict the minimality of $A$.
Therefore, $A$ is an element of $\mc{I}\setminus \mc{I}^*\mc{S}^*$. 
This finishes the proof of our assertion. 
\end{proof}

In the rest of this section we establish an important combinatorial characterization of the ideals of a unipotent numerical monoid. 
We continue with the discussion of a general order theoretic property.
Let $G$ denote either $\mathbf{U}(n,\C)$ or $\mathbf{P}(n,\C)$.
We recall the statement of Theorem~\ref{intro:T4} from the introduction:
\medskip		

Let $\leq$ denote one of the following partial orders: $\leq_{\mc{S},l}$, $\leq_{\mc{S},r}$, or $\leq_{\mc{S},t}$. 
Then $(G(\N),\leq)$ is a locally finite poset.

\begin{proof}[Proof of Theorem~\ref{intro:T4}]
Without loss of generality we will assume that $\leq$ is the order $\leq_{\mc{S},l}$. 
The proofs of the other cases are similar. 
		
Let $[ A,B ]$ be an interval in $(G(\N),\leq)$. 
Then for every $C\in [A,B]$ we have $A\leq C$ and $C\leq B$. 
Equivalently, for every $C\in [A,B]$ we have $A^{-1}C \in \mc{S}$ and $C^{-1}B \in \mc{S}$.
Also, equivalently, for every $C\in [A,B]$ we have $C\in A\mc{S}$ and $B\in C\mc{S}$. 
It follows that if $A,B$, and $C$ are given by $A=(a_{ij})_{1\leq i,j\leq n}$, $B=(b_{ij})_{1\leq i,j\leq n}$, and 
$C=(c_{ij})_{1\leq i,j\leq n}$, then $a_{ij}\leq c_{ij} \leq b_{ij}$ for every $1\leq i,j\leq n$. 
Now, let us assume towards a contradiction that the interval $[A,B]$ has infinitely many elements. 
Then there are infinitely many matrices $C(l):= (c_{ij}(l))_{1\leq i,j\leq n}$, $l=1,2,\dots$ in $[A,B]$ such that $C(l) \in A\mc{S}$ and $B\in C(l)\mc{S}$.
But then there exists a matrix coordinate $(i,j)$, where $1\leq i < j\leq n$, such that $c_{ij}(l) \to \infty$ as $l\to \infty$.  
This means that $b_{ij}\to \infty$ as $l\to \infty$. 
Clearly, this is impossible since the entries of $B$ are fixed integers. 
Therefore, $[A,B]$ has only finitely many elements. 
\end{proof}
\medskip

We proceed to characterize the ideals of $\mc{S}$ by using $\mc{S}$-orders.
\begin{Definition}
Let $(P,\leq)$ be a poset. 
A subset $L\subseteq P$ is called a \red{\it lower order ideal} of $P$ if for every $y\in L$ and $x\in P$, the relation $x\leq y$ implies that $x\in L$. 
\end{Definition}

\begin{Theorem}\label{T:lowerideal}
Let $\mc{I}$ be a subset of $\mc{S}$. 
Then the following statements are equivalent: 
\begin{enumerate}
\item[(1)] $\mc{S}\setminus \mc{I}$ is a lower order ideal of the poset $(\mc{S},\leq_{\mc{S},l})$.
\item[(2)] $\mc{I}$ is a right ideal of $\mc{S}$.
\end{enumerate}
\end{Theorem}
	
\begin{proof}
Let us assume that $\mc{S}\setminus \mc{I}$ is a lower order ideal of $(\mc{S},\leq_{\mc{S},l})$.
This means that if $B\in \mc{S}\setminus \mc{I}$ and $A\in \mc{S}$ are two elements such that $A \leq_{\mc{S},l} B$, then $A$ is an element of $\mc{S}\setminus \mc{I}$ as well. 
Let $X\in \mc{I}$ and $Y\in \mc{S}$. 
We assume towards a contradiction that $XY\notin \mc{I}$. 
Then $XY\in \mc{S}\setminus \mc{I}$ as $\mc{I}\subseteq \mc{S}$. 
Since $Y\in \mc{S}$, the relation $X \leq_{\mc{S},l} XY$ holds. 
Then the lower order ideal property of $\mc{S}\setminus \mc{I}$ implies that $X\in \mc{S}\setminus \mc{I}$, contradicting our initial assumption that $X\in \mc{I}$. 
This shows that (1) implies (2). 
		
Conversely, we assume that $\mc{I}$ is a right ideal of $\mc{S}$.
Let $A\in \mc{S}$ and $B\in \mc{S}\setminus \mc{I}$ be two elements such that $A\leq_{\mc{S},l} B$.
We will show that $A\in \mc{S}\setminus \mc{I}$.
Now, $A\leq_{\mc{S},l} B$ implies that $A^{-1}B\in \mc{S}$ or that $B\in A\mc{S}$.
Of course, if $A$ is an element of $\mc{I}$, then $A\mc{S}\subset \mc{I}$ implies that $B\in \mc{I}$, contradicting our initial assumption that $B\in \mc{S}\setminus \mc{I}$. 
Therefore, we have $A\in \mc{S}\setminus \mc{I}$ as claimed.
This shows that (2) implies (1). 
		
Hence, the proof of our theorem is complete. 
\end{proof}

\begin{Remark}
The proof of our theorem extends naturally to left ideals and two-sided ideals of $\mc{S}$. 
However, for brevity, we omit the statements and proofs here.
\end{Remark}

\section{Frobenius Sets}\label{S:Frobenius}

As we mentioned in the introduction, a noncommutative unipotent numerical monoid has several Frobenius sets.
Recall that the Frobenius sets of a subset $\mc{I}\subseteq \mc{S}$, where $\mc{S}$ is a unipotent numerical monoid, are defined by 
\begin{align*}
\mtt{F}_l(\mc{I}) &:= \{ A \in G(\N) \mid A \notin \mc{I}^* \text{ and } A(G(\N)^*)\subseteq \mc{I}\},\\
\mtt{F}_r(\mc{I}) &:= \{ A \in G(\N) \mid A\notin \mc{I}^* \text{ and } (G(\N)^*)A \subseteq \mc{I}\},\\
\mtt{F}_{t}(\mc{I}) &:= \{ A \in G(\N) \mid A\notin \mc{I}^* \text{ and } A(G(\N)^*)\cup (G(\N)^*)A \subseteq \mc{I}\}.
\end{align*}
In particular, we have 
\begin{align*}
\mtt{F}_{t}(\mc{I}) = \mtt{F}_{l}(\mc{I})\cap \mtt{F}_{r}(\mc{I}).
\end{align*} 
Notice that the Frobenius sets are always non-empty since every maximal element with respect to the entrywise order in $G(\mb{N})\backslash\mc{S}^*$ is an element of the Frobenius set. 
We proceed with another remark.
	
\begin{Remark}\label{R:absurd}
If $A$ and $B$ are two elements from $\mtt{F}_l (\mc{S})$, then there is no $C\in G(\N)^*$ such that $A = BC$.
Indeed, if we assume the existence of such an element $C\in G(\N)^*$, then we find that $BC=A \in G(\N)\setminus \mc{S}$,
which is absurd since $B\in \mtt{F}_l(\mc{S})$.
Likewise, it is easy to see that for $\{A,B\}\subseteq \mtt{F}_r(\mc{S})$ (resp. $\{A,B\}\subseteq \mtt{F}_t(\mc{S})$)
there is no $C\in G(\N)^*$ such that $A=CB$ (resp. there is no $\{C,D\}\subset G(\N)^*$ s.t. $A=BC=DB$).
\end{Remark}

\begin{Example}\label{E:subtleexample0}
We will show in this example that, in general, the left and the right Frobenius sets do not agree. 
Let $\mc{S}$ be the unipotent numerical monoid in $G(\N):= \mathbf{U}(3,\N)$ defined by 
\begin{align*}
\mc{S}=\mtt{N}(\mc{S})\sqcup \mbf{U}(2,\mb{N})_3,
\end{align*}
where $\mtt{N}(\mc{S})$ consists of matrices 
\begin{align*}
\begin{bmatrix}
1 & a & b \\
0 & 1 & c \\
0 & 0 & 1
\end{bmatrix} \in \mathbf{U}(3,\N),
\end{align*}
where $(a,b,c)$ is an element of the set
\begin{align*}
\left\{
\begin{matrix}
(0,0,0), & (1,1,1), & (1,2,1) & (1,2,2) \\  (2,0,1), & (2,0,2), & (2,1,1) &
\end{matrix}
\right\}.
\end{align*}
It is straightforward to check that the right Frobenius set of $\mc{S}$ is given by
\begin{align*}
\mtt{F}_r(\mc{S}) &= \{A\in \mathbf{U}(3,\N) \mid A\notin \mc{S}\text{ and } \mathbf{U}(3,\N)^*A\subset \mc{S}\}\\
&=
\left\{
\begin{bmatrix}
1 & 0 & 2 \\
0 & 1 & 2 \\
0 & 0 & 1
\end{bmatrix}, \begin{bmatrix}
1 & 2 & 2 \\
0 & 1 & 2 \\
0 & 0 & 1
\end{bmatrix}
\right\}.
\end{align*}
The left Frobenius set of $\mc{S}$ is given by
\begin{align*}
\mtt{F}_l(\mc{S}) &= \{A\in \mathbf{U}(3,\N) \mid A\notin \mc{S}\text{ and } A \mathbf{U}(3,\N)^*\subset \mc{S}\}\\
&=
\left\{
\begin{bmatrix}
1 & 2 & 2 \\
0 & 1 & 0 \\
0 & 0 & 1
\end{bmatrix}, 
\begin{bmatrix}
1 & 2 & 2 \\
0 & 1 & 1 \\
0 & 0 & 1
\end{bmatrix}, 
\begin{bmatrix}
1 & 2 & 2 \\
0 & 1 & 2 \\
0 & 0 & 1
\end{bmatrix}
\right\}.
\end{align*}
Clearly, we have $\mtt{F}_r(\mc{S})\neq \mtt{F}_l(\mc{S})$. 
\end{Example}

\begin{Proposition}\label{P:SunionF}
Let $\mc{S}$ be a unipotent numerical monoid.
Let $\mtt{F}$ denote one of $\mtt{F}_l(\mc{S})$, $\mtt{F}_r(\mc{S})$, or $\mtt{F}_{t}(\mc{S})$. 
Then the union $\mc{S} \cup \mtt{F}$ is a unipotent numerical monoid. 
Furthermore, if $\mtt{F}=\mtt{F}_t(\mc{S})$, then $\mc{S}\cup \mtt{C}$ is a unipotent numerical monoid for any subset $\mtt{C}\subseteq\mtt{F}$.
\end{Proposition}

\begin{proof}
We will prove our claim for $\mtt{F}=\mtt{F}_l(\mc{S})$ only. 
The proofs of the other two cases are similar.

Let $A$ and $B$ be two elements from $\mc{S} \cup \mathtt{F}$.
We will show that $AB\in \mc{S} \cup \mathtt{F}$.
We have essentially four cases to consider:
\begin{enumerate}
\item Both $A$ and $B$ are from $\mc{S}$. In this case, we always have $AB\in \mc{S}$. 
\item $A\in \mtt{F}$ and $B\in \mc{S}$. In this case, since $A\in \mtt{F}$, we have $AB\in \mc{S}$.
\item $A\in \mc{S}$ and $B\in \mtt{F}$. In this case, we will show below that $AB\in \mc{S}\cup\mtt{F}$. 
\item Both $A$ and $B$ are from $\mtt{F}$. In this case, we will show below that $AB\in \mc{S}$.
\end{enumerate}
Let us continue with the third case.
For every $C\in G(\N)^*$ we have $BC\in \mc{S}$. 
At the same time, since $A\in \mc{S}$, the membership $BC \in \mc{S}$ implies that $(AB)C = A(BC) \in \mc{S}$.
In other words, we have $AB\in \mathtt{F}$ or $AB\in\mc{S}$. 
Finally, let us handle the fourth case. 
Let $\{A,B\}\subseteq \mathtt{F}$.
Then the assumption $A\in \mathtt{F}$ implies that $AB\in \mc{S}$. 
In conclusion, we see that if $\{A,B\}\subseteq \mc{S}\cup \mathtt{F}$, then $AB \in \mc{S}\cup \mathtt{F}$.
For the last assertion, if $\mtt{F}=\mtt{F}_t(\mc{S})$, then for every $A\in\mc{S}$ and $B,B'\in\mtt{F}$, we have $\{AB,BA,BB'\}\subset \mc{S}\subseteq\mc{S}\cup\mtt{F}$. Hence, $\mc{S}\cup \mtt{C}$ is a unipotent numerical monoid for every subset $\mtt{C}\subseteq\mtt{F}$. 
This finishes the proof of our proposition.
\end{proof}

Frobenius sets are quite useful for characterizing irreducibility.

\begin{Theorem}\label{T:twocharacterizationsofirr}
Let $\mc{S}$ be a unipotent numerical monoid.
Let $\mtt{F}(\cdot )$ denote one of $\mtt{F}_l(\cdot )$, $\mtt{F}_r(\cdot )$, or $\mtt{F}_{t}(\cdot )$. Suppose $|\mtt{F}(\mc{S})|=1$.
Then the following statements are equivalent: 
\begin{enumerate}
\item[(1)] $\mc{S}$ is irreducible.
\item[(2)]  $\mc{S}$ is maximal with respect to set inclusion in the set of unipotent numerical monoids $\mathcal{T}$ such that $\mathtt{F}(\mc{S}) \cap \mathcal{T} = \emptyset$. 
\end{enumerate}
\end{Theorem}

\begin{proof}
We will prove our theorem for $\mtt{F}=\mtt{F}_l$ only. The other cases are proven similarly.

(1)$\Rightarrow$(2). Let $\mathcal{T}$ be a unipotent numerical monoid of $G(\N)$ such that $\mc{T}\cap \mathtt{F}(\mc{S}) = \emptyset$ and $\mc{S}\subseteq \mc{T}$.
Since $\mathcal{T} \cap \mathtt{F}(\mathcal{T}) = \emptyset$, evidently, $\mc{S}$ is given by the intersection $\mathcal{T} \cap (\mc{S} \cup \mathtt{F}(\mc{S}))$. 
By Proposition~\ref{P:SunionF}, the set $\mc{S}\cup \mtt{F}(\mc{S})$ is a unipotent numerical monoid.
Since $\mc{S}\cap \mtt{F}(\mc{S}) = \emptyset$, the irreducibility of $\mc{S}$ implies that $\mathcal{T}=\mc{S}$.
This finishes the proof of our assertion on the maximality of $\mc{S}$.

(2)$\Rightarrow$(1). Let us assume that $\mc{S}$ can be written as the intersection $\mathcal{T}_1\cap \mathcal{T}_2$, where $\mathcal{T}_i$ ($i\in \{1,2\}$) is a unipotent numerical monoid. 
Let $\mtt{F}(\mc{S})=\{C\}$.
Since $C G(\N)^*\subseteq \mc{S} = \mc{T}_1\cap \mc{T}_2$, one of the following holds:
\begin{enumerate}
\item $C\in \mtt{F}(\mc{T}_1) \cap \mtt{F}(\mc{T}_2)$, or 
\item either $C\in \mc{T}_1$ or $C\in \mc{T}_2$, but not both.
\end{enumerate} 
If $C\in \mtt{F}(\mc{T}_1) \cap \mtt{F}(\mc{T}_2)$ holds, then $\mtt{F}(\mc{S}) \cap \mc{T}_i = \emptyset$ for $i\in \{1,2\}$.  
Since $\mc{S}$ is maximal with respect to this property, we find that $\mc{T}_i = \mc{S}$ for some $i\in \{1,2\}$.  
In other words, $\mc{S}$ is irreducible. 
Now, if $C\in \mc{T}_1$ then $\mtt{F}(\mc{S})\cap\mc{T}_2=\emptyset$, which by hypothesis implies $\mc{S}=\mc{T}_2$. 
We can apply a similar argument if $C\in\mc{T}_2$ holds. 
Hence, we see that $\mc{S}$ is irreducible.
This finishes the proof of our theorem. 
\end{proof}

Given the additional assumption $|\mtt{F}(\mc{S})|=1$, we obtain a closely related characterization of irreducible unipotent numerical monoids. 
In the next section, we will show, following the introduction of the `set of special gaps' of a unipotent numerical monoid, that if $\mc{S}$ is irreducible, then $|\mtt{F}(\mc{S})|=1$.

\begin{Theorem}\label{T:threecharacterizationsofirr3}
Let $\mc{S}$ be a unipotent numerical monoid.
Let $\mtt{F}(\mc{S}) \in \{ \mtt{F}_l(\mc{S}),\mtt{F}_r(\mc{S}),\mtt{F}_t(\mc{S})\}$.
If $|\mtt{F}(\mc{S})|=1$ holds, then the following statements are equivalent: 
\begin{enumerate}
\item[(1)] $\mc{S}$ is irreducible.
\item[(2)]  $\mc{S}$ is maximal with respect to set inclusion in the set of unipotent numerical monoids $\mathcal{T}$ such that
$\mathtt{F}(\mathcal{T}) = \mathtt{F}(\mc{S})$.  
\item[(3)]  $\mc{S}$ is maximal with respect to set inclusion in the set of unipotent numerical monoids $\mathcal{T}$ such that
$\mathtt{F}(\mc{S}) \cap \mathcal{T} = \emptyset$. 
\end{enumerate}
\end{Theorem}

\begin{proof}
First, we prove our theorem for $\mtt{F}=\mtt{F}_l$. 
The proof of the case of $\mtt{F}=\mtt{F}_r$ is very similar, so, we omit it.

(1)$\Rightarrow$(2). Let $\mathcal{T}$ be a unipotent numerical monoid of $G(\N)$ such that $\mathtt{F}(\mathcal{T}) = \mathtt{F}(\mc{S})$. Let us assume that $\mathcal{T}\supseteq \mc{S}$. 
Since $\mathcal{T} \cap \mathtt{F}(\mathcal{T}) = \emptyset$, evidently, $\mc{S}$ is given by the intersection $\mathcal{T} \cap (\mc{S} \cup \mathtt{F}(\mc{S}))$. 
By Proposition~\ref{P:SunionF}, the set $\mc{S}\cup \mtt{F}(\mc{S})$ is a unipotent numerical monoid. 
But then the irreducibility of $\mc{S}$ implies that $\mathcal{T}=\mc{S}$.
This shows that $\mc{S}$ is maximal.

(2)$\Rightarrow$(3). Let us assume that $\mc{S}$ is contained in a unipotent numerical monoid $\mathcal{T}$ such that $\mathtt{F}(\mc{S}) \cap \mathcal{T} = \emptyset$. We will show that $\mc{S}=\mc{T}$. 

First, let us assume that there exists $A\in \mtt{F}(\mc{S})\setminus \mtt{F}(\mc{T})$.
This means that $AG(\N)^*\subset \mc{S}$ but there exists $C\in G(\N)^*$ such that $AC \in \mc{S}\setminus \mc{T}$. 
Since $\mc{S}$ is a subset of $\mc{T}$ and since $AC$ is an element of $\mc{S}$, the last statement is absurd. 
In other words, we have $\mtt{F}(\mc{S})\setminus \mtt{F}(\mc{T})=\emptyset$, implying that $\mtt{F}(\mc{S})\subseteq \mtt{F}(\mc{T})$.

Next, let us assume that $A\in \mtt{F}(\mc{T})\setminus \mtt{F}(\mc{S})$.
This means that $AG(\N)^*\subset \mc{T}$ but there exists $C\in G(\N)^*$ such that $AC \in \mc{T}\setminus \mc{S}$. 
If for some $E\in G(\N)^*$ the element $ACE$ is still not in $\mc{S}$, then we replace $C$ by $C E$.
Then we repeat this procedure continually. 
Since $G(\N)\setminus \mc{S}$ is a finite set, our procedure will come to an end after a finite number of replacements. 
The resulting element, which we still denote by $AC$ has the property that $AC\notin \mc{S}$ but $AC G(\mb{N})^*\subseteq \mc{S}$.
Said differently, we found an element $AC$ of $\mtt{F}(\mc{S})$. 
But we already know that $\mtt{F}(\mc{S})\subseteq \mtt{F}(\mc{T})$. 
By Remark~\ref{R:absurd} (2), this is absurd since the elements of $\mtt{F}(\mc{T})$ can not be obtained from each other by multiplying by an element from $G(\N)^*$. 
Hence, we conclude that $\mtt{F}(\mc{S})= \mtt{F}(\mc{T})$. 
Now, our assumption on the maximality of $\mc{S}$ implies that $\mc{S}= \mc{T}$. 
This finishes the proof of (2)$\Rightarrow$(3).

(3)$\Rightarrow$(1). We already know this implication from Theorem~\ref{T:twocharacterizationsofirr}.
Hence, the proof of our theorem is complete for $\mtt{F} \in \{ \mtt{F}_l, \mtt{F}_r\}$. 
\medskip

We now proceed under the assumption that $\mtt{F}=\mtt{F}_t$. 
The proofs of the implications (1)$\Rightarrow$(2) and (3)$\Rightarrow$(1) remain valid in this case. 
We are left with proving the assertion (2)$\Rightarrow$(3).
This part of the proof will be a subtle adaptation of the arguments used for the similar implication in the case of $\mtt{F}=\mtt{F}_l$.
To distinguish the Frobenius sets here, we will employ subscripts.

Let us assume that $\mc{S}$ is contained in a unipotent numerical monoid $\mathcal{T}$ such that $\mathtt{F}_t(\mc{S}) \cap \mathcal{T} = \emptyset$. 
We will show that $\mc{S}=\mc{T}$. 

First, let us assume that there exists $A\in \mtt{F}_t(\mc{S})\setminus \mtt{F}_t(\mc{T})$.
Since $A\notin \mtt{F}_t(\mc{T})$ we know that either $A\notin \mtt{F}_l(\mc{T})$ or $A\notin \mtt{F}_r(\mc{T})$.
Without loss of generality, we may assume that $A\notin \mtt{F}_l(\mc{T})$. 
This means that there exists $C\in G(\N)^*$ such that $AC \notin \mc{T}$.
Now, since $A\in \mtt{F}_t(\mc{S}) \subset \mtt{F}_l(\mc{S})$, we know that $A$ is an element of $G(\N) \setminus \mc{S}^*$ such that $AG(\N)^*\subseteq \mc{S}$. In particular, we know that $AC\in \mc{S}$. 
But this is absurd since we assumed that $\mc{S}\subseteq \mc{T}$.
In conclusion, we have $\mtt{F}_t(\mc{S})\subseteq \mtt{F}_t(\mc{T})$.
We emphasize our previous observation that the inclusion $\mtt{F}_l(\mc{S})\subseteq \mtt{F}_l(\mc{T})$ holds as well.
This reiteration will be useful in the next paragraph.

Next, let us assume that $A\in \mtt{F}_t(\mc{T})\setminus \mtt{F}_t(\mc{S})$.
By using the argument as in the previous paragraph, we may assume without loss of generality that $A\in \mtt{F}_l(\mc{T})\setminus \mtt{F}_l(\mc{S})$.
This means that 1) $AG(\N)^*\subseteq \mc{T}$, and 2) there exists $C\in G(\N)^*$ such that $AC \notin \mc{S}$.
If for some $E\in G(\N)^*$ the element $ACE$ is still not in $\mc{S}$, then we replace $C$ by $C E$.
Then we repeat this procedure continually. 
Since $G(\N)\setminus \mc{S}$ is a finite set, our procedure will come to an end after a finite number of replacements. 
The resulting element, which we still denote by $AC$ has the property that $AC\notin \mc{S}$ but $AC G(\mb{N})^*\subseteq \mc{S}$.
Said differently, we found an element $AC$ of $\mtt{F}_l(\mc{S})$. 
But we already know that $\mtt{F}_l(\mc{S})\subseteq \mtt{F}_l(\mc{T})$. 
By Remark~\ref{R:absurd} (2), this is absurd since the elements of $\mtt{F}_l(\mc{T})$ can not be obtained from each other by multiplying by an element from $G(\N)^*$. 
Hence, we conclude that there is no such $A$. 
In other words, we obtained the equality $\mtt{F}_t(\mc{S})= \mtt{F}_t(\mc{T})$. 
Now, our assumption on the maximality of $\mc{S}$ implies that $\mc{S}= \mc{T}$. 
This finishes the proof of (2)$\Rightarrow$(3).
\end{proof}

We now present a sufficient condition for the irreducibility of a unipotent numerical monoid that will be useful when we investigate the symmetric and pseudo-symmetric unipotent numerical monoids. 
We note that this theorem was stated as Theorem~\ref{intro:T5} in the introduction.

\begin{Theorem}\label{T:irreducibility}
Let $\mc{S}$ be a unipotent numerical monoid.
If one of the following conditions hold true, then $\mc{S}$ is irreducible: 
\begin{enumerate}
\item[(1)] $|\mtt{F}_l(\mc{S})|=1$ and for every $A\in G(\N)\setminus \mc{S}$ we have either $\mtt{F}_l(\mc{S}) \cap A \mc{S} \neq \emptyset$ or
$\mtt{F}_l(\mc{S}) \cap  \mc{S} A\neq \emptyset$; 
\item[(2)] $|\mtt{F}_r(\mc{S})|=1$ and for every $A\in G(\N)\setminus \mc{S}$ we have either $\mtt{F}_r(\mc{S}) \cap \mc{S}A \neq \emptyset$ or 
$\mtt{F}_r(\mc{S}) \cap A\mc{S} \neq \emptyset$;
\item[(3)] $|\mtt{F}_t(\mc{S})|=1$ and for every $A\in G(\N)\setminus \mc{S}$ we have either $\mtt{F}_t(\mc{S}) \cap \mc{S}A \neq \emptyset$ or 
$\mtt{F}_t(\mc{S}) \cap  A\mc{S} \neq \emptyset$.
\end{enumerate}
\end{Theorem}
\begin{proof}
We will prove (1) only. The other assertions are proven similarly. 
		
Let $\mc{S}$ be a unipotent numerical monoid such that $|\mtt{F}_l(\mc{S})|=1$ and for every $A\in G(\N)\setminus \mc{S}$ 
we have either $\mtt{F}_l(\mc{S}) \cap A \mc{S} \neq \emptyset$ or $\mtt{F}_l(\mc{S}) \cap  \mc{S} A\neq \emptyset$.
We assume towards a contradiction that $\mc{S}$ is reducible. 
Then, by Theorem~\ref{T:twocharacterizationsofirr}, there exists a unipotent numerical monoid $\mc{T}$, properly containing $\mc{S}$, such that $\mc{T}\cap \mtt{F}_l(\mc{S}) = \emptyset$. 
Let $A\in \mc{T}\setminus \mc{S}$.
Then, by the hypothesis of our theorem, we have either $\mtt{F}_l(\mc{S}) \cap A \mc{S} \neq \emptyset$ or
$\mtt{F}_l(\mc{S}) \cap  \mc{S} A\neq \emptyset$.
In other words, there exists $B\in \mc{S}$ such that either $AB\in \mtt{F}_l(\mc{S})$ or $BA\in \mtt{F}_l(\mc{S})$ holds. 
Since $\mc{S}$ is a subset of $\mc{T}$ we see that both $A$ and $B$ are from $\mc{T}$.
It follows that either $AB\in \mc{T}$ or $BA\in \mc{T}$ holds. 
But this contradicts our assumption that $\mc{T}\cap \mtt{F}_l(\mc{S})=\emptyset$.
Therefore, $\mc{S}$ is an irreducible unipotent numerical monoid. 
\end{proof}

\begin{Example}
Let $\mc{S}\subset \mathbf{P}(3,\N)$ be the unipotent numerical monoid defined as in Figure~\ref{F:S2}.
The points colored in blue represent the minimal generators.
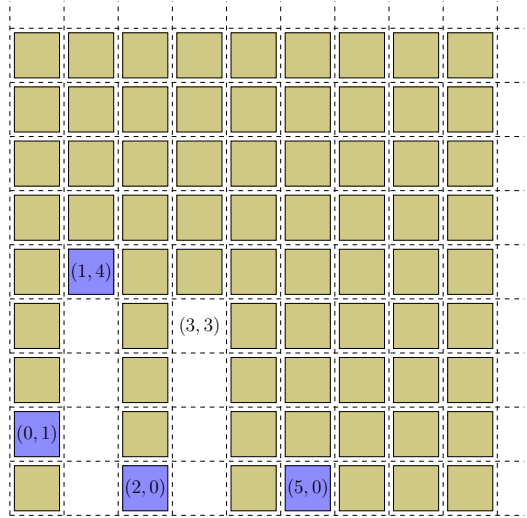
\begin{figure}[htp]
\begin{center}
\scalebox{.6}{
\begin{tikzpicture}[scale=1.2]

\draw (0,0)  node[fill=olive!45!, minimum size=1cm,draw] {};
\draw (0,1)  node[fill=blue!45!, minimum size=1cm,draw] {};
\node at (0,1) {$(0,1)$};

\foreach \i in {4,...,8} {\draw (\i,1)  node[fill=olive!45!, minimum size=1cm,draw] {};};
\foreach \i in {4,...,8} {\draw (\i,2)  node[fill=olive!45!, minimum size=1cm,draw] {};};
\foreach \i in {4,...,8} {\draw (\i,3)  node[fill=olive!45!, minimum size=1cm,draw] {};};
\foreach \i in {0,...,8} {\draw (\i,4)  node[fill=olive!45!, minimum size=1cm,draw] {};};
\foreach \i in {0,...,8} {\draw (\i,5)  node[fill=olive!45!, minimum size=1cm,draw] {};};
\foreach \i in {0,...,8} {\draw (\i,6)  node[fill=olive!45!, minimum size=1cm,draw] {};};
\foreach \i in {0,...,8} {\draw (\i,7)  node[fill=olive!45!, minimum size=1cm,draw] {};};
\foreach \i in {0,...,8} {\draw (\i,8)  node[fill=olive!45!, minimum size=1cm,draw] {};};
\foreach \i in {1,...,3} {\draw (2,\i)  node[fill=olive!45!, minimum size=1cm,draw] {};};
\foreach \i in {6,...,8} {\draw (\i,0)  node[fill=olive!45!, minimum size=1cm,draw] {};};

\draw (0,2)  node[fill=olive!45!, minimum size=1cm,draw] {}; 
\draw (0,3)  node[fill=olive!45!, minimum size=1cm,draw] {}; 
\draw (4,0)  node[fill=olive!45!, minimum size=1cm,draw] {}; 
\draw (5,0)  node[fill=blue!45!, minimum size=1cm,draw] {}; 
\draw (2,0)  node[fill=blue!45!, minimum size=1cm,draw] {}; 
\draw (1,4)  node[fill=blue!45!, minimum size=1cm,draw] {};

\node at (1,4) {$(1,4)$};
\node at (2,0) {$(2,0)$};
\node at (5,0) {$(5,0)$};
\node at (3,3) {$(3,3)$};

\foreach \i in {0,...,9} {\draw [dashed] (\i-.5,-.5) -- (\i-.5,9);};
\foreach \i in {0,...,9} {\draw [dashed] (-.5,\i-0.5) -- (9,\i-0.5);};

\end{tikzpicture}
}
\end{center}
\caption{An irreducible, commutative unipotent numerical monoid in $\mathbf{P}(3,\Z)$.}
\label{F:S2}
\end{figure}

Then we have 
\begin{align*}
\mtt{F}_l(\mc{S}) = \{ (3,3) \}\ \text{ and } \ 
\mtt{Gap}(\mc{S}) = \{ (1,0),(1,1),(1,2),(1,3), (3,0),(3,1), (3,2),(3,3)\}.
\end{align*}
Now we check the condition (1) in Theorem~\ref{T:irreducibility}:
\begin{align*}
(1,3) + (2,0) &= (3,3),  \\
(1,2) + (2,1) &= (3,3),  \\
(1,1) + (2,2) &= (3,3),  \\
(1,0) + (2,3) &= (3,3),  \\
(3,0) + (0,3) &= (3,3),  \\
(3,1) + (0,2) &= (3,3),  \\
(3,2) + (0,1) &= (3,3), \\
(3,3) + (0,0) &= (3,3).
\end{align*}
In these equalities, the second summands on the left hand side are from $\mc{S}$ while the first summands are from $\mtt{Gap}(\mc{S})$.  
Hence, according to Theorem~\ref{T:irreducibility}, $\mc{S}$ is an irreducible unipotent numerical monoid. 
\end{Example}

We will show by an example that the converse of Theorem~\ref{T:irreducibility} need not be true. 
\begin{Example}\label{E:conversefails}
Let $\mathcal{S} \subset \mathbf{P}(3,\mathbb{N})$ be the commutative unipotent numerical monoid defined as in Figure~\ref{F:S22}. 
The points colored in blue are the minimal generators of the monoid.
In this example, the Frobenius set of $\mathcal{S}$ is given by $\mathtt{F}_t(\mathcal{S}) = \{ (2,2) \}$. Let $A := (1,1)$. 
It is easy to check that neither $\mathtt{F}_t(\mathcal{S}) \cap A \mathcal{S} \neq \emptyset$ nor $\mathtt{F}_t(\mathcal{S}) \cap \mathcal{S} A \neq \emptyset$ holds. 
Hence, part (3) of Theorem~\ref{T:irreducibility} fails. 
Note that since $\mathcal{S}$ is commutative, we do not need to check parts (1) and (2) of Theorem~\ref{T:irreducibility}. 
This provides an example where the converse of Theorem~\ref{T:irreducibility} does not hold. 
Nevertheless, in this case, $\mathcal{S}$ is irreducible. 
There are several ways to see this. 
A direct calculation can be performed, and one method involves using the ``special gaps" of a generalized numerical semigroup, as described in~\cite[Proposition 2.5]{CFPU2019}. Another useful approach is presented in~\cite[Theorem 5.7]{CFPU2019}. 
Given that the genus of $\mathcal{S}$ is 5, we have $2\mathtt{g}(\mathcal{S}) - 1 = 9$.
At the same time, the unique element $F$ of the Frobenius set of $\mathcal{S}$ is given by $F = (F^{(1)}, F^{(2)}) = (2,2)$. 
Since $(F^{(1)} + 1)(F^{(2)} + 1) = 9 = 2\mathtt{g}(\mathcal{S}) - 1$, according to~\cite[Theorem 5.7]{CFPU2019}, $\mathcal{S}$ is ``pseudo-symmetric'' and hence irreducible. 
(Later in this article, we will introduce pseudo-symmetric unipotent numerical semigroups, which generalize pseudo-symmetric generalized numerical semigroups.)

\begin{figure}[htp]
\begin{center}
\scalebox{.6}{\begin{tikzpicture}[scale=1.2]

\draw (0,0)  node[fill=olive!45!, minimum size=1cm,draw] {};

\foreach \i in {3,...,5} {\draw (\i,0)  node[fill=olive!45!, minimum size=1cm,draw] {};};
\foreach \i in {3,...,5} {\draw (\i,1)  node[fill=olive!45!, minimum size=1cm,draw] {};};
\foreach \i in {3,...,5} {\draw (\i,2)  node[fill=olive!45!, minimum size=1cm,draw] {};};
\foreach \i in {0,...,5} {\draw (\i,3)  node[fill=olive!45!, minimum size=1cm,draw] {};};
\foreach \i in {0,...,5} {\draw (\i,4)  node[fill=olive!45!, minimum size=1cm,draw] {};};
\foreach \i in {0,...,5} {\draw (\i,5)  node[fill=olive!45!, minimum size=1cm,draw] {};};

\draw (0,5)  node[fill=blue!45!, minimum size=1cm,draw] {}; 
\draw (0,4)  node[fill=blue!45!, minimum size=1cm,draw] {}; 
\draw (0,3)  node[fill=blue!45!, minimum size=1cm,draw] {}; 
\draw (2,0)  node[fill=blue!45!, minimum size=1cm,draw] {}; 
\draw (2,0)  node[fill=blue!45!, minimum size=1cm,draw] {}; 
\draw (2,1)  node[fill=blue!45!, minimum size=1cm,draw] {}; 
\draw (1,2)  node[fill=blue!45!, minimum size=1cm,draw] {};
\draw (1,3)  node[fill=blue!45!, minimum size=1cm,draw] {};
\draw (1,4)  node[fill=blue!45!, minimum size=1cm,draw] {};
\draw (3,0)  node[fill=blue!45!, minimum size=1cm,draw] {};
\draw (3,1)  node[fill=blue!45!, minimum size=1cm,draw] {};
						
\node at (2,0) {$(2,0)$};
\node at (2,1) {$(2,1)$};
\node at (3,1) {$(3,1)$};
\node at (3,0) {$(3,0)$};
\node at (1,2) {$(1,2)$};
\node at (1,3) {$(1,3)$};
\node at (1,4) {$(1,4)$};
\node at (0,3) {$(0,3)$};
\node at (0,4) {$(0,4)$};
\node at (0,5) {$(0,5)$};
\node at (2,2) {\color{purple}$(2,2)$};
\node at (1,1) {\color{gray}$(1,1)$};

\foreach \i in {0,...,6} {\draw [dashed] (\i-.5,-.5) -- (\i-.5,5);};
\foreach \i in {0,...,6} {\draw [dashed] (-.5,\i-0.5) -- (5.5,\i-0.5);};
\end{tikzpicture}}
\end{center}
\caption{The converse of Theorem~\ref{T:irreducibility} fails.}
\label{F:S22}
\end{figure}
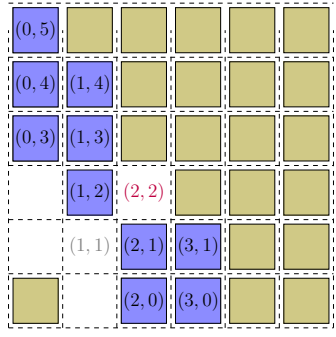
\end{Example}

\section{Pseudo-Frobenius Sets}\label{S:Pseudo}

Let $\mc{S}$ be a unipotent numerical monoid.
Let $\mc{I}$ be a relative right ideal of $\mc{S}$. 
We now introduce the \red{\em left pseudo-Frobenius set of $\mc{I}$}:
\begin{align*}
\mathtt{PF}_l(\mc{I}) := \{A \in G(\N) \mid A\notin \mc{I}^* \text{ and } A\mc{S}^* \subseteq \mc{I}\}.
\end{align*}
The \red{\em right pseudo-Frobenius set of $\mc{I}$} is defined by:
\begin{align*}
\mathtt{PF}_r(\mc{I}) := \{A \in G(\N) \mid A\notin \mc{I}^* \text{ and } \mc{S}^*A \subseteq \mc{I}\}.
\end{align*}
Finally, we define the \red{\em two-sided pseudo-Frobenius set of $\mc{I}$} here. 
\begin{align*}
\mathtt{PF}_{t}(\mc{I}) := \{A \in G(\N) \mid A\notin \mc{I}^*\ \text{ and } \{ AC,CA\} \subseteq \mc{I} \ \text{ for all }\ C\in \mc{S}^*\}.
\end{align*}

The left (resp. right, two-sided) pseudo-Frobenius set of $\mc{S}$ as a unipotent numerical monoid is defined as the left (resp. right, two-sided)  pseudo-Frobenius set of $\mc{S}$ as a right ideal of itself.

\begin{Example}\label{E:linearly}
Let $\mc{S}$ denote the commutative unipotent numerical monoid defined as in Figure~\ref{F:S3}.
The points colored in blue are the minimal generators of $\mc{S}$. 
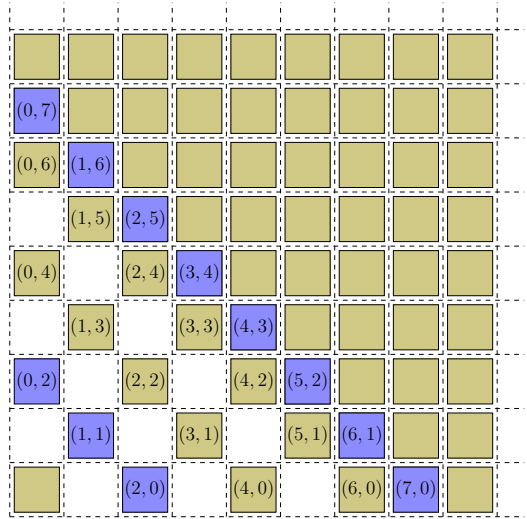
\begin{figure}[htp]
\begin{center}
\scalebox{.6}{
\begin{tikzpicture}[scale=1.2]

\draw (0,0)  node[fill=olive!45!, minimum size=1cm,draw] {};

\foreach \i in {6,...,8} {\draw (\i,0)  node[fill=olive!45!, minimum size=1cm,draw] {};}; 
\foreach \i in {5,...,8} {\draw (\i,1)  node[fill=olive!45!, minimum size=1cm,draw] {};};
\foreach \i in {4,...,8} {\draw (\i,2)  node[fill=olive!45!, minimum size=1cm,draw] {};};
\foreach \i in {3,...,8} {\draw (\i,3)  node[fill=olive!45!, minimum size=1cm,draw] {};};
\foreach \i in {2,...,8} {\draw (\i,4)  node[fill=olive!45!, minimum size=1cm,draw] {};};
\foreach \i in {1,...,8} {\draw (\i,5)  node[fill=olive!45!, minimum size=1cm,draw] {};};
\foreach \i in {0,...,8} {\draw (\i,6)  node[fill=olive!45!, minimum size=1cm,draw] {};};
\foreach \i in {0,...,8} {\draw (\i,7)  node[fill=olive!45!, minimum size=1cm,draw] {};};
\foreach \i in {0,...,8} {\draw (\i,8)  node[fill=olive!45!, minimum size=1cm,draw] {};};
\foreach \i in {0,...,4} {\draw (4-\i,\i)  node[fill=olive!45!, minimum size=1cm,draw] {};};
\foreach \i in {0,...,2} {\draw (2-\i,\i)  node[fill=blue!45!, minimum size=1cm,draw] {};};
\foreach \i in {0,...,7} {\draw (7-\i,\i)  node[fill=blue!45!, minimum size=1cm,draw] {};};

\node at (0,2) {$(0,2)$};
\node at (1,1) {$(1,1)$};
\node at (2,0) {$(2,0)$};

\node at (4,0) {$(4,0)$};
\node at (3,1) {$(3,1)$};
\node at (2,2) {$(2,2)$};
\node at (1,3) {$(1,3)$};
\node at (0,4) {$(0,4)$};

\node at (6,0) {$(6,0)$};
\node at (5,1) {$(5,1)$};
\node at (4,2) {$(4,2)$};
\node at (3,3) {$(3,3)$};
\node at (2,4) {$(2,4)$};
\node at (1,5) {$(1,5)$};
\node at (0,6) {$(0,6)$};
\node at (7,0) {$(7,0)$};
\node at (6,1) {$(6,1)$};
\node at (5,2) {$(5,2)$};
\node at (4,3) {$(4,3)$};
\node at (3,4) {$(3,4)$};
\node at (2,5) {$(2,5)$};
\node at (1,6) {$(1,6)$};
\node at (0,7) {$(0,7)$};

\foreach \i in {0,...,9} {\draw [dashed] (\i-.5,-.5) -- (\i-.5,9);};
\foreach \i in {0,...,9} {\draw [dashed] (-.5,\i-0.5) -- (9,\i-0.5);};

\end{tikzpicture}}

\end{center}
\caption{A commutative unipotent numerical monoid in $\mathbf{P}(3,\N)$.}
\label{F:S3}
\end{figure}
It is easy to check that the pseudo-Frobenius set of $\mc{S}$ is given by 
\begin{align*}
\mtt{PF}_t(\mc{S}) = \{ (0,5), (1,4), (2,3), (3,2), (4,1), (5,0) \},
\end{align*}
which is also the Frobenius set of $\mc{S}$. 
\end{Example}

The pseudo-Frobenius sets can be characterized as follows.
\begin{Proposition}\label{P:leftrightPFS}
We maintain our previous notation.  
Then we have 
\begin{enumerate}
\item[(1)] $\mathtt{PF}_l(\mc{I}) = \max_{\leq_{\mc{S},l}} (G(\N) \setminus \mc{I}^*)$,
\item[(2)] $\mathtt{PF}_r(\mc{I}) = \max_{\leq_{\mc{S},r}} (G(\N) \setminus \mc{I}^*)$,
\item[(3)] $\mathtt{PF}_{t}(\mc{I}) = \max_{\leq_{\mc{S},l}} (G(\N) \setminus \mc{I}^*)\cap \max_{\leq_{\mc{S},r}} (G(\N) \setminus \mc{I}^*)$.
\end{enumerate}
\end{Proposition}

Before proving our proposition, we have a preliminary observation.
\begin{Lemma}\label{L:leftrightintheintersection}
Let $\mc{I}$ be a relative right ideal of $\mc{S}$.
Then we have the inclusion
\begin{align*}
\max_{\leq_{\mc{S},t}} (G(\N) \setminus \mc{I}^*)\  \supseteq  \ 
\max_{\leq_{\mc{S},l}} (G(\N) \setminus \mc{I}^*)\  \cap  \ \max_{\leq_{\mc{S},r}} (G(\N) \setminus \mc{I}^*).
\end{align*}
\end{Lemma}

\begin{proof}
Let $A\in \max_{\leq_{\mc{S},l}} (G(\N) \setminus \mc{I}^*)\cap \max_{\leq_{\mc{S},r}} (G(\N) \setminus \mc{I}^*)$. 
Let $B$ be an element of $G(\mb{N})\backslash \mc{I}^*$ such that $A\leq_{\mc{S},t} B$. 
Then, since $\{BA^{-1},A^{-1}B\}\subseteq \mc{S}$, we have the relations $A\leq_{\mc{S},l} B$ and $A\leq_{\mc{S},r} B$, implying that $B=A$.
\end{proof}

\begin{proof}[Proof of Proposition~\ref{P:leftrightPFS}]
		
(1) Let $A\in \mathtt{PF}_l(\mc{I})$.
Then $AD\in \mc{I}$ for all $D\in \mc{S}^*$.
Let $\mathbf{B}$ denote the set of elements $B\in G(\N) \setminus \mc{I}$ such that $A\leq_{\mc{S},l} B$.
This is nonempty since $A\in \mathbf{B}$. 
We assume towards a contradiction that there exists $B\in \mathbf{B} \setminus \{A\}$.
The relation $A\lneq_{\mc{S},l} B$ implies that there exists $C\in \mc{S}^*$ such that $A^{-1}B = C$.
Since $B= AC$, and since $A\in \mtt{PF}_l(\mc{I})$, we see that $B\in \mc{I}$. 
But this contradicts with our initial assumption that $B\notin \mc{I}$. 
Therefore, $\mathbf{B} = \{A\}$.  
This observation shows that $\mathtt{PF}_l(\mc{I})$ is contained in the set of maximal elements of $G(\N)\setminus \mc{I}$ with respect to the partial order $\leq_{\mc{S},l}$.

Conversely, let $A\in\max_{\leq_{\mc{S},l}}(G(\N)\setminus \mc{I})$. 
We will show that $A\in \mathtt{PF}_l(\mc{I})$. 
Equivalently, we will show that $A\mc{S}^* \subseteq \mc{I}$.
Let us assume towards a contradiction that there exists $B \in \mc{S}^*$ such that $AB \notin \mc{I}$.
We will compare $AB$ with $A$.
Since $A^{-1}AB = B \in \mc{S}$, we see that $A\leq_{\mc{S},l} AB$. 
But $A$ is a maximal element. 
It follows that $A=AB$ or equivalently that $\mathbf{1}_n = B$ since $G(\N)$ is a cancellative monoid. 
This contradicts with the fact that $B$ is from $\mc{S}^*$. 
Hence, $A$ is an element of $\mathtt{PF}_l(\mc{I})$ as claimed. 
This finishes the proof of our first assertion.
\\
		
(2) The proof of this case is almost identical to the proof of (1) so we skip it. 
\\

(3) Let $A\in \mathtt{PF}_t(\mc{I})$.
Then $\{AD,DA\}\subset \mc{I}$ for every $D\in \mc{S}^*$.
Let $\mathbf{B}_l$ and $\mbf{B}_r$ denote the sets of elements $B_l,B_r\in G(\N) \setminus \mc{I}$ such that $A\leq_{\mc{S},l} B_l$ and $A\leq_{\mc{S},r} B_r$, respectively.
These are nonempty sets since $A\in \mathbf{B}_l$ and $A\in\mbf{B}_r$. 
We will show that $\mbf{B}_l=\mathbf{B}_r=\{A\}$. 
Now, assume that $B\in\mbf{B}_l\backslash \{A\}$. 
Then $A^{-1}B=C\in\mc{S}$, implying that $B=AC\in \mc{I}$.
This contradicts with the fact $B\notin\mc{I}$. 
Hence, we have $\mbf{B}_l=\{A\}$. 
Similarly, we see that $\mbf{B}_r=\{A\}$. 
Thus we showed that $A\in \max_{\leq_{\mc{S},l}} (G(\N) \setminus \mc{I})\cap \max_{\leq_{\mc{S},r}} (G(\N) \setminus \mc{I})$.

Conversely, let $A\in \max_{\leq_{\mc{S},l}} (G(\N) \setminus \mc{I})\cap \max_{\leq_{\mc{S},r}} (G(\N) \setminus \mc{I})$. 
We will show that $A\in \mathtt{PF}_{t}(\mc{I})$. 
Equivalently, we will show that $A\mc{S}^*\cup \mc{S}^*A \subseteq \mc{I}$.
Let us assume towards a contradiction that there exists $B \in \mc{S}^*$ such that either $AB \notin \mc{I}$ or $BA \notin \mc{I}$.
Without loss of generality, we assume that $AB\notin \mc{I}$. 
We will compare $AB$ with $A$ but in the order $\leq_{\mc{S},l}$ not in the order $\leq_{\mc{S},t}$.
Since $A^{-1}AB = B \in \mc{S}$, we see that $A\leq_{\mc{S},l} AB$. 
But $A$ is a maximal element of $\max_{\leq_{\mc{S},l}} (G(\N) \setminus \mc{I})$ by assumption.
It follows that $A=AB$ or equivalently that $\mathbf{1}_n = B$ since $G(\N)$ is a cancellative monoid. 
This contradicts with the fact that $B$ is from $\mc{S}^*$. 
Hence, $A$ is an element of $\mathtt{PF}_{t}(\mc{I})$ as claimed. 
This finishes the proof of our proposition.
\end{proof}

\begin{Proposition}\label{P:PFisinF}
Let $\mc{I}$ be a relative right ideal of $\mc{S}$. 
Then we have the following assertions:
\begin{enumerate}
\item[(1)] $\mathtt{PF}_{t}(\mc{I}) = \mathtt{PF}_l(\mc{I}) \cap \mathtt{PF}_r(\mc{I})$,
\item[(2)] $\mathtt{F}_l(\mc{I}) \subseteq  \mathtt{PF}_l(\mc{I})$,
\item[(3)] $\mathtt{F}_r(\mc{I}) \subseteq  \mathtt{PF}_r(\mc{I})$,
\item[(4)] $\mathtt{F}_t(\mc{I}) \subseteq  \mathtt{PF}_t(\mc{I})$.
\end{enumerate}
\end{Proposition}
\begin{proof}
The first item follows immediately from Proposition~\ref{P:leftrightPFS}.  
The assertions (2) and (3) follow from the definitions. 
Finally, (4) follows from (1), (2), and (3) and the definition of the two-sided Frobenius set.
\end{proof}
\medskip

We now investigate the relationship between the pseudo-Frobenius sets and the Ap\'ery sets. 
Recall that the left and the right Ap\'ery sets of an element $A\in \mc{S}^*$ are defined by 
\begin{align*}
\mtt{Ap}_l(\mc{S}, A) := \{ B\in \mc{S} \mid A^{-1}B \notin \mc{S} \} \quad \text{and}\quad \mtt{Ap}_r(\mc{S}, A) := \{ B\in \mc{S} \mid BA^{-1} \notin \mc{S} \}.
\end{align*}

\begin{Theorem}\label{T:PF_Apery}
Let $A\in \mc{S}^*$. Then we have 
\begin{align*}
\mtt{PF}_{t}(\mc{S})=  A^{-1} (\max_{\leq_{\mc{S},l}} \mtt{Ap}_l(\mc{S},A)) \cap  (\max_{\leq_{\mc{S},r}} \mtt{Ap}_r(\mc{S},A) ) A^{-1}.
\end{align*}
\end{Theorem}
	
\begin{proof}
Let $Z\in \mtt{PF}_{t}(\mc{S})$. 
Then, for $A\in \mc{S}^*$, we have $ZA\in \mc{S}$ and $AZ \in \mc{S}$.  
Since $A^{-1}AZ= Z \notin \mc{S}$, we see that $AZ\in \mtt{Ap}_l(\mc{S},A)$. 
Likewise, since $ZAA^{-1} = Z \notin \mc{S}$, we see that $ZA \in \mtt{Ap}_r (\mc{S},A)$. 

We proceed to show that $AZ$ is a maximal element of $\mtt{Ap}_l(\mc{S},A)$ with respect to $\leq_{\mc{S},l}$. 
This will show that $Z$ is an element of $A^{-1} \max_{\leq_{\mc{S},l}} \mtt{Ap}_l(\mc{S},A)$. 
Let $Y\in  \mtt{Ap}_l(\mc{S},A)$ be such that $AZ \lneq_{\mc{S},l} Y$.
This means that $Y=AZC$ for some $C\in \mc{S}^*$. 
Since $A^{-1}Y\notin \mc{S}$, we see that $ZC\notin \mc{S}$. 
But this contradicts with our initial assumption that $Z\in \mtt{PF}_{t}(\mc{S}) \subset \mtt{PF}_l(\mc{S})$. 
Hence, $AZ$ is a maximal element of $\mtt{Ap}_l(\mc{S},A)$ with respect to $\leq_{\mc{S},l}$. 
In particular, this argument shows that $Z\in A^{-1} \max_{\leq_{\mc{S},l}} \mtt{Ap}_l(\mc{S},A)$.

Next, we will show that $ZA$ is a maximal element of $\mtt{Ap}_r(\mc{S},A)$ with respect to $\leq_{\mc{S},r}$. 
This will show that $Z$ is an element of $\max_{\leq_{\mc{S},r}} \mtt{Ap}_r(\mc{S},A) A^{-1}$. 
Let $Y\in  \mtt{Ap}_r(\mc{S},A)$ be such that $ZA \lneq_{\mc{S},r} Y$.
This means that $Y=CZA$ for some $C\in \mc{S}^*$. 
Since $YA^{-1} \notin \mc{S}$, we see that $CZ\notin \mc{S}$. 
But this contradicts with our initial assumption that $Z\in \mtt{PF}_{t}(\mc{S}) \subseteq \mtt{PF}_r(\mc{S})$. 
Hence, $ZA$ is a maximal element of $\mtt{Ap}_r(\mc{S},A)$ with respect to $\leq_{\mc{S},r}$. 
In particular, this argument shows that $Z\in (\max_{\leq_{\mc{S},r}} \mtt{Ap}_r(\mc{S},A)) A^{-1}$.
In conclusion, we showed that 
\begin{align*}
\mtt{PF}_{t}(\mc{S})\subseteq  A^{-1} (\max_{\leq_{\mc{S},l}} \mtt{Ap}_l(\mc{S},A)) \cap  (\max_{\leq_{\mc{S},r}} \mtt{Ap}_r(\mc{S},A) ) A^{-1}.
\end{align*}

Conversely, let $Z \in A^{-1} \max_{\leq_{\mc{S},l}} \mtt{Ap}_l(\mc{S},A)$.
This means that $AZ \in  \max_{\leq_{\mc{S},l}} \mtt{Ap}_l(\mc{S},A)$.
Hence, $AZ \in\mc{S}$ and $Z=A^{-1}(AZ) \notin\mc{S}$. 
Assume that $Z$ is not an element of $\mtt{PF}_l(\mc{S})$. 
Then there is an element $C\in \mc{S}^*$ such that $ZC \notin \mc{S}$. 
Since $AZ \in \mc{S}$ and $C\in \mc{S}^*$, we see that $AZC$ is an element of $\mc{S}$. 
At the same time, we know that $A^{-1}AZC=ZC$ is not an element of $\mc{S}$.
It follows that $AZC \in \mtt{Ap}_l(\mc{S},A)$. 
But $AZ \leq_{\mc{S},l} AZC$ since $(AZ)^{-1} AZC = C \in \mc{S}$.
Since $AZ$ is a maximal element of $\mtt{Ap}_l(\mc{S},A)$, we see that $AZ = AZC$, hence $\mathbf{1}_n = C$.
This contradicts with the fact that $C\in \mc{S}^*$. 
It follows that $Z\in \mtt{PF}_l(\mc{S})$. 
A similar argument shows that $Z\in \mtt{PF}_r(\mc{S})$. 
Hence, we proved that 
\begin{align*}
\mtt{PF}_{t}(\mc{S})\supseteq  A^{-1} (\max_{\leq_{\mc{S},l}} \mtt{Ap}_l(\mc{S},A)) \cap  (\max_{\leq_{\mc{S},r}} \mtt{Ap}_r(\mc{S},A) ) A^{-1}.
\end{align*}
This finishes the proof of our theorem.
\end{proof}

The proof of our previous theorem gives the proof of the following result, also.
\begin{Theorem}
Let $A\in \mc{S}^*$. Then we have 
\begin{align*}
\mtt{PF}_{l}(\mc{S})=  A^{-1} (\max_{\leq_{\mc{S},l}} \mtt{Ap}_l(\mc{S},A)) \quad \text{ and }\quad 
\mtt{PF}_{r}(\mc{S})= (\max_{\leq_{\mc{S},r}} \mtt{Ap}_r(\mc{S},A) ) A^{-1}.
\end{align*}
\end{Theorem}

\begin{Definition}
Let $\mc{S}$ be a unipotent numerical monoid. The \red{\em set of special gaps} of $\mc{S}$ is defined by
\begin{align*}
\mtt{SG}(\mc{S}) := \{ A\in \mtt{PF}_{t}(\mc{S}) \mid A^2 \in \mc{S} \}.
\end{align*}
\end{Definition}

\begin{Lemma}\label{L:specialgaps+1}
Let $A\in G(\N) \setminus \mc{S}$. Then $A\in \mtt{SG}(\mc{S})$ if and only if $\mc{S} \cup \{A\}$ is a unipotent numerical monoid.
\end{Lemma}
	
\begin{proof}
If $A\in \mtt{SG}(\mc{S})$, then $A\in \mtt{PF}_{t}(\mc{S})$. 
Hence, we already know that $\{AX,XA\}\subset \mc{S}$ for every $X\in \mc{S}^*$.
At the same time, $A\in \mtt{SG}(\mc{S})$ implies that $A^2\in \mc{S}$.
These observations show that $\mc{S} \cup \{A\}$ is closed under multiplication. 
Since $A$ is from $G(\N)$, we proved that $\mc{S} \cup \{A\}$ is a unipotent numerical monoid. 
		
Conversely, if $\mc{S} \cup \{A\}$ is a unipotent numerical monoid, it is a cancellative semigroup.
In particular, we have $A\mc{S}^*\subseteq \mc{S}$ and $\mc{S}^*A\subseteq \mc{S}$. 
It follows that $A$ is an element of $\mtt{PF}_{t}(\mc{S})$.
\end{proof}

We are now ready to prove our Theorem~\ref{intro:T6} from the introduction. 
Since it has an additional part, we state it as a theorem to be able to give a reference to it later.

\begin{Theorem}(Theorem~\ref{intro:T6})\label{C:irreducibleimpliesF=1}
If $\mc{S}$ is an irreducible unipotent numerical monoid, then $\mtt{F}_t(\mc{S})=\mtt{SG}(\mc{S})$ and $|\mtt{F}_t(\mc{S})|=1$. 
\end{Theorem}

\begin{proof}
We begin with the proof of $\mtt{F}_t(\mc{S})=\mtt{SG}(\mc{S})$.
It follows from definitions that $\mtt{F}_t(\mc{S})\subseteq \mtt{PF}_t(\mc{S})$. 
Furthermore, for every $A\in \mtt{F}_t(\mc{S})$ we have $A^2 \in \mc{S}$. 
In other words, every element of $\mtt{F}_t(\mc{S})$ is an element of $\mtt{SG}(\mc{S})$. 
For the reverse inclusion, for $D\in\mtt{SG}(\mc{S})\backslash\mtt{F}_t(\mc{S})$, by Lemma~\ref{L:specialgaps+1}, 
$\mc{S}\cup \{D\}$ is a unipotent numerical monoid such that $(\mc{S}\cup \{D\}) \cap \mtt{F}_t(\mc{S})=\emptyset$. 
Then we have $\mc{S}=(\mc{S}\cup \{D\})\cap(\mc{S}\cup \mtt{F}_t(\mc{S}))$, which contradicts with the irreducibility of $\mc{S}$.

We proceed to show that $|\mtt{F}_t(\mc{S})|=1$. 
If $\mtt{F}_t(\mc{S})$ contains two distinct elements then we find two distinct unipotent numerical monoids $\mc{T}_1 := \mc{S}\sqcup \{A_1\}$ and $\mc{T}_2 := \mc{S} \sqcup \{A_2\}$, where $\{A_1,A_2\}\subseteq \mtt{F}_t(\mc{S})$ such that $\mc{T}_1 \cap \mc{T}_2 = \mc{S}$.
This contradicts with the irreducibility of $\mc{S}$. 
Therefore, we conclude that $|\mtt{F}_t(\mc{S})|=1$. 
This finishes the proof of our theorem. 
\end{proof}

We conclude this section by relating the irreducibility of $\mc{S}$ to its torsion monoid and the set of special gaps. 
We will prove our Theorem~\ref{intro:T2} from the introduction, which states the following: 
\medskip

Let $\mc{S}$ be a unipotent numerical monoid.
Then $\mc{S}$ is irreducible if and only if the following two conditions hold:
\begin{enumerate}
\item[(1)] the torsion monoid $Rel_T(\mc{S},\subseteq)$ has a unique minimal nontrivial idempotent, and 
\item[(2)] $\mtt{F}_t(\mc{S})=\mtt{SG}(\mc{S})$.
\end{enumerate}

\begin{proof}[Proof of Theorem~\ref{intro:T2}]
($\Rightarrow$) 
Let $\mc{S}$ be an irreducible unipotent numerical monoid. 
By Theorem~\ref{C:irreducibleimpliesF=1}, we know that $\mtt{F}_t(\mc{S})=\mtt{SG}(\mc{S})$.
We proceed to show (1). 
Let $E(Rel_T(\mc{S},\subseteq))$ denote the set of all idempotents of $Rel_T(\mc{S},\subseteq)$.

We now assume that $\mc{S}$ is an irreducible unipotent numerical monoid. 
Towards a contradiction, let us assume that $E(Rel_T(\mc{S},\subseteq))$ has two nontrivial distinct minimal idempotents, denoted by $\mc{T}_1$ and $\mc{T}_2$. 
Since we have $\mc{T}_2\not\subseteq \mc{T}_1$ and $\mc{T}_1\not\subseteq \mc{T}_2$, there exist $A_1\in\mc{T}_1\setminus \mc{T}_2$ and $A_2\in\mc{T}_2\setminus \mc{T}_1$. For these elements, we have $A_1 \mc{S} \cup \mc{S} A_1\subseteq \mc{T}_1$ and $A_2 \mc{S} \cup \mc{S} A_2 \subseteq \mc{T}_2$. It follows that $A_1\notin A_2 \mc{S} \cup \mc{S} A_2$ and $A_2\notin A_1 \mc{S} \cup \mc{S} A_1$.

For $A\in G(\N)$, let us denote by $\langle \mc{S},A\rangle$ the monoid generated by the set $\mc{S}\cup A$.   
Observe that $\mc{S}\subseteq \mc{S} \cup \{A\} \subseteq \langle \mc{S},A\rangle$ and that $\langle \mc{S},A\rangle$ is a submonoid of $G(\N)$.
So, $\langle \mc{S},A\rangle$ is a unipotent numerical monoid and also $\langle \mc{S},A\rangle \in Rel_T(\mc{S},\subseteq )$.

We apply these observations to the monoids $\langle \mc{S},A_1\rangle$ and $\langle \mc{S},A_2\rangle$.
By Lemma~\ref{L:unmiffidempotent}, these monoids are contained in the set of idempotents, $E(Rel_T(\mc{S},\subseteq))$. 
In fact, since $A_i \notin \mc{S}$ ($i\in \{1,2\}$), they are nontrivial idempotents. 
Moreover, it follows from our previous observations on the properties of the elements $A_1$ and $A_2$ that 
\begin{align*}
\mc{S}=\langle\mc{S},A_1\rangle\cap\langle\mc{S},A_2\rangle.
\end{align*} 
Since $\langle\mc{S}, A_2\rangle \subseteq \mc{T}_2$, by the minimality of $\mc{T}_2$, we have $\langle\mc{S}, A_2\rangle = \mc{T}_2$.
By the same argument, we have $\langle\mc{S}, A_1\rangle = \mc{T}_1$.
But now this contradicts with the fact that $\mc{S}$ is irreducible. 
Hence, we conclude that there is only one minimal nontrivial idempotent. 
\medskip
		
($\Leftarrow$) 
We assume that the torsion monoid has a unique minimal nontrivial idempotent and that $\mtt{F}_t(\mc{S})=\mtt{SG}(\mc{S})$ holds. 
If $\mtt{F}_t(\mc{S})$ contains two distinct elements $C_1$ and $C_2$, then by using Remark~\ref{R:absurd} and Proposition~\ref{P:SunionF} we find two distinct minimal nontrivial idempotents $\mc{S}\cup \{C_1\}$ and $\mc{S}\cup \{C_2\}$ where $C_1,C_2\in \mtt{F}_t(\mc{S})$, which is absurd. 
Hence, we know that $|\mtt{F}_t(\mc{S})|=1$. 
It follows that $\mc{S}\cup \mtt{F}_t(\mc{S})$ is the unique minimal nontrivial idempotent in $E(Rel_T(\mc{S},\subseteq))$. 

We now assume towards a contradiction that $\mc{S}$ is not irreducible.
Then there exist unipotent numerical monoids $\mc{T}_1,\dots, \mc{T}_r$ such that 
$\mc{S}\subsetneq \mc{T}_i \subsetneq G(\N)$ for every $i\in \{1,\dots, r\}$ and $\bigcap_{i=1}^r \mc{T}_i = \mc{S}$. 
But since $\mc{S}\cup \mtt{F}_t(\mc{S})$ is the unique minimal nontrivial idempotent, we have 
$$
\mc{S}\cup \mtt{F}_t(\mc{S}) \subseteq  \mc{T}_i \ \text{ for every $i\in \{1,\dots, r\}$},
$$
implying
$$
\mc{S} \neq \mc{S}\cup \mtt{F}_t(\mc{S}) \subseteq \bigcap_{i=1}^r \mc{T}_i.
$$
This contradicts with our initial assumption that $\mc{S}$ is given by the intersection $ \bigcap_{i=1}^r \mc{T}_i$. 
Hence, we conclude that $\mc{S}$ is irreducible. 
This finishes the proof of our theorem.
\end{proof}

\section{Irreducibility and the Pseudo-Frobenius Elements}\label{S:Irreducibility}

Let $\mc{S}$ be a unipotent numerical monoid. 
We define \red{\em left (resp. right, resp. two-sided) type} of $\mc{S}$ as the cardinality of $\mtt{PF}_l(\mc{S})$ (resp. of $\mtt{PF}_r(\mc{S})$, resp. of $\mtt{PF}_t(\mc{S})$). 
Notice that if 
\begin{align*}
|\mtt{PF}_l(\mc{S})| = |\mtt{PF}_r(\mc{S})|= 1
\end{align*}
then we have $|\mtt{PF}_t(\mc{S})|=1$ also. 
But the converse of this implication is not necessarily true as we will show in our next example.

\begin{Example}\label{E:subtleexample1}
Let $\mc{S}$ be the unipotent numerical monoid in $G(\N):= \mathbf{U}(3,\N)$ such that
\[
\mc{S}=\mtt{N}(\mc{S})\cup \mbf{U}(2,\mb{N})_2,
\]
where 
$$
\mtt{N}(\mc{S})
=
\left\{ 
\begin{bmatrix}
1 & 0 & 0 \\
0 & 1 & 0 \\
0 & 0 & 1
\end{bmatrix},
\begin{bmatrix}
1 & 1 & 1 \\
0 & 1 & 0 \\
0 & 0 & 1
\end{bmatrix},
\begin{bmatrix}
1 & 0 & 1 \\
0 & 1 & 1 \\
0 & 0 & 1
\end{bmatrix},
\begin{bmatrix}
1 & 1 & 0 \\
0 & 1 & 1 \\
0 & 0 & 1
\end{bmatrix}
\right\}.
$$
Then the gap set of $\mc{S}$ is given by 
$$
\mtt{Gaps}(\mc{S})
=
\left\{ 
\begin{bmatrix}
1 & 0 & 0 \\
0 & 1 & 1 \\
0 & 0 & 1
\end{bmatrix},
\begin{bmatrix}
1 & 0 & 1 \\
0 & 1 & 0 \\
0 & 0 & 1
\end{bmatrix},
\begin{bmatrix}
1 & 1 & 0 \\
0 & 1 & 0 \\
0 & 0 & 1
\end{bmatrix},
\begin{bmatrix}
1 & 1 & 1 \\
0 & 1 & 1 \\
0 & 0 & 1
\end{bmatrix}
\right\}.
$$
The following facts are easy to check:
\begin{enumerate}
\item $\mc{S}$ is irreducible.
\item $\mtt{F}_l(\mc{S})  = \mtt{F}_r(\mc{S}) =\mtt{F}_t(\mc{S}) =
\left\{\begin{bmatrix}
1 & 1 & 1 \\
0 & 1 & 1 \\
0 & 0 & 1
\end{bmatrix}
\right\}$.

\item $\mtt{PF}_r(\mc{S})= 
\left\{\begin{bmatrix}
1 & 0 & 0 \\
0 & 1 & 1 \\
0 & 0 & 1
\end{bmatrix},
\begin{bmatrix}
1 & 1 & 1 \\
0 & 1 & 1 \\
0 & 0 & 1
\end{bmatrix}
\right\}$.
\item $\mtt{PF}_l(\mc{S})=
\left\{
\begin{bmatrix}
1 & 1 & 0 \\
0 & 1 & 0 \\
0 & 0 & 1
\end{bmatrix}, 
\begin{bmatrix}
1 & 1 & 1 \\
0 & 1 & 1 \\
0 & 0 & 1
\end{bmatrix}
\right\}$.
\end{enumerate}
In conclusion, there are irreducible unipotent numerical monoids $\mc{S}$ such that 
\begin{align*}
\mtt{PF}_l(\mc{S}) \neq \mtt{PF}_r (\mc{S})\quad\text{and}\quad |\mtt{PF}_l(\mc{S}) \cap \mtt{PF}_r (\mc{S})|=1.
\end{align*}
\end{Example}

We present another example to motivate our definition of a symmetric unipotent numerical monoid.

\begin{Example}\label{E:subtleexample2}
For the second example let $\mc{S}$ be the unipotent numerical monoid in $G(\N):= \mathbf{U}(3,\N)$ such that
\[\mc{S}=\mtt{N}(\mc{S})\cup \mbf{U}(2,\mb{N})_4\]
where $\mtt{N}(\mc{S})$ consists of
\begin{align*}
\begin{bmatrix}
1 & a & b \\
0 & 1 & c \\
0 & 0 & 1
\end{bmatrix} \in \mathbf{U}(3,\N),
\end{align*}
where $(a,b,c)$ is an element of the set
\begin{align*}
\mtt{N}(\mc{S}):=
\left\{
\begin{matrix}
(0,0,0), & (0,1,1), & (0,2,2), & (0,3,3) & (1,0,2), & (1,1,1), & (1,1,2), & (1,1,3), \\
(1,2,1), & (1,2,2), & (1,2,3), & (1,3,1), & (1,3,2), & (1,3,3), & (2,0,1), & (2,0,2), \\
(2,1,2), & (2,1,3), & (2,2,1), & (2,2,2), & (2,2,3), & (2,3,2), & (2,3,3), & (3,0,3), \\
(3,1,1), & (3,1,2), & (3,1,3), & (3,2,1), & (3,2,2), & (3,2,3), & (3,3,2), & (3,3,3)
\end{matrix}
\right\}.
\end{align*}
Let us use the triplet notation to represent the elements of $\mtt{Gaps}(\mc{S})$ as well.
Hence, we have 
\begin{align*}
\mtt{Gaps}(\mc{S}):=
\left\{
\begin{matrix}
(0,0,1), & (0,0,2), & (0,0,3), & (0,1,0), & (0,1,2), & (0,1,3), & (0,2,0), & (0,2,1), \\
(0,2,3), & (0,3,0), & (0,3,1), & (0,3,2), & (1,0,0), & (1,0,1), & (1,0,3), & (1,1,0), \\ 
(1,2,0), & (1,3,0), & (2,0,0), & (2,0,3), & (2,1,0), & (2,1,1), & (2,2,0), & (2,3,0), \\
(2,3,1), & (3,0,0), & (3,0,1), & (3,0,2), & (3,1,0), & (3,2,0), & (3,3,0), & (3,3,1)
\end{matrix}
\right\}.
\end{align*}
Then we have 
\begin{align}\label{A:NG}
|\mtt{N} (\mc{S})|= |\mtt{Gaps}(\mc{S}) |.
\end{align}
At the same time, it is easy to check that $|\mtt{F}_t(\mc{S})|$ has more than one element.
Hence, by Theorem~\ref{intro:T6}, we see that $\mc{S}$ can not be irreducible.
\end{Example}
	
It is noteworthy that for a numerical monoid $S \subseteq \mathbb{N}$, if the analogous condition to (\ref{A:NG}) holds true, that is,
\[
\mtt{n}(S) = \mtt{g}(S),
\]
then the equality $\mtt{g}(S) = \frac{\mtt{F}(S) + 1}{2}$ also holds true. 
These equivalent conditions imply that $S$ is an irreducible numerical monoid. 
In fact, this is how ``symmetric numerical monoids" are defined. 
Our example shows that such a property must be defined with greater precision.

\medskip

In the rest of this section we discuss the refinements of the set of irreducible unipotent numerical monoids. 
We begin with recalling the definitions of symmetric and pseudo-symmetric unipotent numerical monoids from the introduction: 
\begin{itemize}
\item An irreducible unipotent numerical monoid is called \red{\em symmetric} if for every $A\in G(\N)\setminus \mc{S}$ we have 
$$
\mtt{F}_t(\mc{S}) \cap A\mc{S} \neq \emptyset \quad\text{or}\quad   \mtt{F}_t(\mc{S})\cap \mc{S}A \neq \emptyset.
$$
\item An irreducible unipotent numerical monoid is called \red{\em pseudo-symmetric} if there exists $B\in G(\N)\setminus \mc{S}$ such that $B^2 \in \mtt{F}_t(\mc{S})$ and for every $A\in G(\N)\setminus (\mc{S} \sqcup \{B\})$ we have 
$\mtt{F}_t(\mc{S}) \cap A\mc{S} \neq \emptyset$ or $\mtt{F}_t(\mc{S})\cap \mc{S}A \neq \emptyset$.
\end{itemize}

\begin{Remark}
If $\mc{S}$ is symmetric, then there is no $B\in G(\N)\setminus \mc{S}$ such that $B^2 \in \mtt{F}_t(\mc{S})$.
Otherwise, it follows from the defining condition that we have either $B^2 \in \{BD , DB\}$ for some $D\in \mc{S}$.
This is absurd since $\mc{S}$ is a cancellative monoid.  
\end{Remark}
	
\begin{Proposition}\label{P:PF=F}
Let $\mc{S}$ be a symmetric unipotent numerical monoid. Then we have $|\mtt{PF}_t(\mc{S})| = |\mtt{F}_t(\mc{S})|=1$. 
\end{Proposition}

\begin{proof}
Let $A \in \mtt{PF}_t(\mc{S})$. 
Then $A\in G(\N)\setminus \mc{S}$ and $(\mc{S}^*A \cup A\mc{S}^*) \subseteq \mc{S}^*$.
Since $\mc{S}$ is symmetric, we have either $\mtt{F}_t(\mc{S}) \cap A\mc{S} \neq \emptyset$ or $\mtt{F}_t(\mc{S})\cap \mc{S}A \neq \emptyset$.
Hence, unless $A\in \mtt{F}_t(\mc{S})$, we have $\mtt{F}_t(\mc{S})\cap \mc{S}^* \neq \emptyset$.
But (the left, the right, or the two-sided) Frobenius subsets of $\mc{S}$ are always disjoint from $\mc{S}$. 
This contradiction shows that the membership $A\in \mtt{F}_t(\mc{S})$ holds true. 
It follows that $\mtt{PF}_t(\mc{S})$ is a subset of $\mtt{F}_t(\mc{S})$.
Since the converse inclusion holds automatically, we find that $\mtt{PF}_t(\mc{S}) = \mtt{F}_t(\mc{S})$.
Finally, since $\mc{S}$ is irreducible, by Theorem~\ref{intro:T6}, we have $|\mtt{F}_t(\mc{S})| = 1$. 
This finishes the proof of our proposition.
\end{proof}
\medskip

In an earlier version of our paper, we had used more stringent definitions for symmetric and pseudo-symmetric unipotent numerical monoids. 
An anonymous referee suggested that these earlier definitions be referred to by a different name.
In light of the referee's insightful suggestion, we now introduce the ``strong'' versions of symmetric and pseudo-symmetric unipotent numerical monoids.

\begin{Definition}
We call an irreducible unipotent numerical monoid \red{\em strongly symmetric} if for every $A\in G(\N)\setminus \mc{S}$ we have 
$\mtt{F}_t(\mc{S}) \cap A\mc{S} \cap \mc{S}A \neq \emptyset$.
We call an irreducible unipotent numerical monoid \red{\em strongly pseudo-symmetric} if the following two conditions are satisfied:
\begin{enumerate}
\item there exists an element $B\in G(\N)\setminus \mc{S}$ such that $B^2 \in \mtt{F}_t(\mc{S})$, and 
\item for every $A\in G(\N)\setminus (\mc{S} \sqcup \{B\})$ we have 
$\mtt{F}_t(\mc{S}) \cap A\mc{S} \cap \mc{S}A \neq \emptyset$.
\end{enumerate}
\end{Definition}

For this definition, we have the following result. 
	
\begin{Theorem}\label{intro:T7'}
Let $\mc{S}$ be a unipotent numerical monoid. 
If $\mc{S}$ is irreducible, then $\mc{S}$ is either strongly symmetric or strongly pseudo-symmetric.
\end{Theorem}

\begin{proof}
Since $\mc{S}$ is irreducible, we know that $|\mtt{F}_t(\mc{S}) | = 1$. 
Let $C$ denote the unique element of $\mtt{F}_t(\mc{S})$. 
We define 
\begin{align*}
\mc{D} := \{ H \in G(\N) \setminus \mc{S} \mid  C\neq H^2 \text{ and } C \notin H\mc{S}^*\cap \mc{S}^*H\}.
\end{align*}
It is easy to check that $C\in \mc{D}$. 
Let us assume that $\mc{D}\setminus \{C\}\neq \emptyset$. 
Let $H$ be an element from $\max_{\leq_e} (\mc{D}\setminus \{C\})$. 
We will show that $\mc{S}\sqcup \{H\}$ is a unipotent numerical monoid. 
Let $A$ and $B$ be two elements from $\mc{S}\sqcup \{H\}$. 
We check all mutually distinct cases: 
\begin{enumerate}
\item[(a)] We assume that both $A$ and $B$ are from $\mc{S}$. 
Then we automatically have $AB\in \mc{S}\sqcup \{H\}$.

\item[(b)] We assume that $A\in \mc{S}^*$ and $B= H$. 
We will show that $BA \in \mc{S}$. 
Of course, if $BA$ is already an element of $\mc{S}$, then there is nothing to do. 
Towards a contradiction, we proceed with the assumption that $BA\notin \mc{S}$. 
Since $B$ is a maximal element with respect to $\leq_e$, we see that $BA\notin \mc{D}$. 
Hence we have one of the following conditions:
\begin{enumerate}
\item[(1)] $C = (BA)^2$, 
\item[(2)] $C\in (BA)\mc{S}^*\cap  \mc{S}^*(BA)$.
\end{enumerate} 

We proceed with (2). Then, since $A \in \mathcal{S}^*$ and $C\in (BA)\mc{S}^*\cap  \mc{S}^*(BA)$, we have $C \in B \mathcal{S}^*$.
Let $E, F \in \mathcal{S}^*$ be the elements such that $C = BE = FBA$. 
Since $B$ is from $\mathcal{D}$, we know that $C \notin B \mathcal{S}^* \cap \mathcal{S}^* B$.  
Therefore, $C\in B\mathcal{S}^*$ implies that $C \notin \mathcal{S}^* B$. 
In particular, we observe that $C \notin \mathcal{S}^* FB \cap FB \mathcal{S}^*$.
We now have two additional observations regarding the element $FB$. 
First, we observe that $(FB)^2 \neq C$. Indeed, if we assume the contrary, then $FBFB = FBA$ implies that $FB = A$. It follows that $FBA = A^2$, which is absurd. 
Secondly, we observe that $FB \notin \mathcal{S}$. 
Otherwise, we would have $FBA \in \mathcal{S}$, which is absurd.
These observations imply that $FB \in \mathcal{D}$. 
However, this leads to a contradiction since $B$ is maximal with respect to $\leq_e$. Hence, item (2) is not possible.

We proceed with (1). 
Let $X:=BAB$.
Then $X\notin\mc{S}^*$.
Otherwise, we would have $C=XA \in \mc{S}^*$, which is absurd. 
It follows from the cancellative property of $\mc{S}$ that $X^2 \neq C$. 
We proceed to check whether the following membership holds or not:
\begin{align}\label{A:trueornot}
C \in \mc{S}^* X \cap X \mc{S}^*.
\end{align} 
Let us assume that (\ref{A:trueornot}) is true.
Then we have an element $E\in \mc{S}^*$ such that 
\begin{align}\label{A:baba=abab}
C = EBAB= BABA.
\end{align}
We now observe that both $Z:= EBA$ and $Y:= ABA$ cannot be elements of $\mc{S}^*$ at the same time.  
Otherwise, we see that $C = B(ABA) \in B\mc{S}^*$ and $C = (EBA)B \in \mc{S}^*B$.
This means that $C\in B\mc{S}^* \cap \mc{S}^* B$, which is not possible since $B$ is from $\mc{D}$. 
Let us proceed with the assumption that $Y \notin \mc{S}^*$.
Evidently, $Y^2\neq C$. 
It is also true that $C\notin Y\mc{S}^* \cap \mc{S}^* Y$. 
Otherwise, we find that $C= LY$ for some $L\in \mc{S}^*$.
By the cancellative property of $\mc{S}$, we see that $B=L$, which is absurd. 
But all of these observations show that $Y\in \mc{D}$, which contradicts with our assumption that $B$ is maximal with respect to $\leq_e$. 
Thus, we may now assume that $Z\notin \mc{S}^*$. 
Notice that $Z^2 = EBAEBA$. 
If $Z^2 = C$, then by the cancellative property of $\mc{S}$, we see that $EBAE= BA$, which is absurd. 
Finally, it is also true that $C\notin Z\mc{S}^* \cap \mc{S}^* Z$. 
Otherwise, we find that $C= LZ=LEBA$ for some $L\in \mc{S}^*$. 
By using the cancellative property of $\mc{S}$ once more, we see that $LE= BA$.
Since both $L$ and $E$ are from $\mc{S}^*$, we see that $LE \in \mc{S}^*$. 
However, by our initial assumption $BA \notin \mc{S}^*$. 
Hence, we obtained our desired contradiction that item (1) is not possible neither.

\item[(c)] We assume that $A\in \mc{S}^*$ and $B= H$. 
We will show that $AB \in \mc{S}$.
Once again, towards a contradiction, we assume that $AB\notin \mc{S}$. 
As in part (b), the maximality of $B$ implies that one of the following conditions must hold true:
\begin{enumerate}
\item[(1)] $C = (AB)^2$, 
\item[(2)] $C\in (AB)\mc{S}^*\cap  \mc{S}^*(AB)$.
\end{enumerate} 

We begin with (2). Then, since $A \in \mathcal{S}^*$ and $C\in (AB)\mc{S}^*\cap  \mc{S}^*(AB)$, we have $C \in \mathcal{S}^*B$.
Let $E, F \in \mathcal{S}^*$ be the elements such that $C = EB = ABF$. 
Since $B$ is from $\mathcal{D}$, we know that $C \notin B \mathcal{S}^* \cap \mathcal{S}^* B$.  
Therefore, $C\in \mathcal{S}^*B$ implies that $C \notin B\mathcal{S}^* $.

We will make some observations regarding the element $BF$. 
First, we observe that $(BF)^2 \neq C$. Indeed, if we assume the contrary, then $BFBF = ABF$ implies that $BF = A$. 
It follows that $C=ABF = A^2$, which is absurd. 
Secondly, we observe that $BF \notin \mathcal{S}$. 
Otherwise, we would have $ABF \in \mathcal{S}$, which is absurd.
Finally, we check if $C\in \mc{S}^* BF \cap BF \mc{S}^*$ or not. 
Since $BF\mc{S}^*\subset B\mc{S}^*$ and since we observed earlier that $C \notin B\mathcal{S}^* $,
we see that $C\notin \mc{S}^* BF \cap BF \mc{S}^*$.
These observations imply that $BF \in \mathcal{D}$. 
However, this leads to a contradiction since $B$ is maximal with respect to $\leq_e$. Hence, item (2) is not possible.

We now consider (1). 
Let $Y:=BAB$. 
Then $Y\notin \mc{S}^*$.
Otherwise, we find that $C=AY\in \mc{S}^*$, which is absurd. 
It is also evident that $Y^2 \neq C$. 
We proceed to check whether $C\in Y\mc{S}^* \cap \mc{S}^*Y$ holds or not. 
Let us assume that it is true.
Then there exist $E\in \mc{S}^*$ such that 
$$
ABAB = C =  YE = BABE.
$$
We claim that both $W:=ABA$ and $Z:=ABE$ cannot be elements of $\mc{S}^*$ at the same time.
Otherwise, we see that $C = (ABA)B \in \mc{S}^*B$ and $C=B(ABE) \in B\mc{S}^*$, implying that 
$C\in \mc{S}^*B\cap B\mc{S}^*$. 
This is absurd, since $B$ is from $\mc{D}$. 

We proceed with the assumption that $W\notin \mc{S}^*$. 
Clearly, $W^2\neq C$. 
It is also true that $C\notin W\mc{S}^* \cap \mc{S}^* W$.
Otherwise, there exists $L\in \mc{S}^*$ such that $ABAB = ABAL$.
By the cancellative property of $\mc{S}$, we see that $B=L$ holds, which is absurd. 
These observations imply that $W\in \mc{D}$.
But this contradicts with the maximality of $B$ with respect to $\leq_e$. 

We now proceed with the assumption that $Z\notin \mc{S}^*$. 
Note that $Z^2\neq C$. Otherwise we would have $ABEABE = ABAB$, implying $EABE=AB$, which is absurd. 
Notice also that $C\notin Z\mc{S}^* \cap \mc{S}^* Z$. 
Otherwise, we find that $C= ZL = ABEL$. 
Then, by the cancellative property of $\mc{S}$, we have $AB=EL\in \mc{S}^*$, which contradicts with our initial assumption that $AB\notin \mc{S}^*$. 
Hence, $C\notin Z\mc{S}^* \cap \mc{S}^* Z$. 
Now all of these observations show that $Z\in \mc{D}$.
But, once again, this leads to a contradiction with the maximality of $B$ with respect to $\leq_e$. 
Hence, (1) is not possible also.

\item[(d)] We assume that $A=B= H$. 
We will show that $H^2\in \mc{S}$.
			
Towards a contradiction, we assume that $H^2 \notin \mc{S}$. 
Parts (b) and (c) show that for every $A\in \mc{S}^*$ we have $\{ AH , HA \} \subseteq \mc{S}^*$.
Hence, it follows from the definition of the two-sided pseudo-Frobenius sets that $H\in \mtt{PF}_t(\mc{S})$. 
At the same time, since $H$ is a maximal element, we know that $H^2 \notin \mc{D}$. 
This means that either $H^4 = C$ or $C\in H^2\mc{S}^* \cap \mc{S}^* H^2$. 
If $H^4=C$ holds, then we claim that $H^3\in \mc{D}$ has to hold.
Let us assume that $H^3\in\mc{S}$. 
Then $C=H^4$ is an element of $H\mc{S}^*\cap \mc{S}^*H$, which is absurd. 
Hence, we proceed with the assumption that $H^3\in \mtt{G}(\mb{N})\backslash\mc{S}$. 
Now, if $H^3$ is not an element of $\mc{D}$, then we must have that $C=H^4\in H^3\mc{S}^*\cap \mc{S}^* H^3$.
Then it follows that $H\in \mc{S}^*$, which is a contradiction. 
Hence, we proved our claim that $H^3$ is contained in $\mc{D}$ if $C=H^4$ is assumed. 
But this contradicts with the maximality of $H$ in $\mc{D}\setminus \{C\}$. 
Therefore, the remaining possibility, that is, $C\in H^2\mc{S}^*\cap \mc{S}^* H^2$ must be true. 
In this case, let $A,B\in\mc{S}^*$ be two elements such that $C=AH^2=H^2B$. 
By parts (b) and (c), we have $AH,HB\in \mc{S}^*$.
It follows that $C\in H\mc{S}^*\cap \mc{S}^* H$, a contradiction. 
In conclusion, we have $H^2\in \mc{S}$. 
\end{enumerate}

Hence, we showed that $\mc{S} \sqcup \{H\}$ is a unipotent numerical monoid. 
We now have two unipotent numerical monoids $\mc{S} \sqcup \{C\}$ and $\mc{S}\sqcup \{H\}$, properly containing $\mc{S}$, such that we have 
$$
(\mc{S} \sqcup \{C\}) \cap (\mc{S}\sqcup \{H\}) = \mc{S}.
$$ 
This contradicts with our hypothesis that $\mc{S}$ is irreducible.  
We proceed with the assumption that $\mc{D} = \{ C\}$. 
It follows from the definition of $\mc{D}$ that for every $H \in G(\N) \setminus \mc{S} \sqcup \{C\}$ such that $H^2\neq C$ we have $C\in H\mc{S}^*\cap \mc{S}^*H$.
Equivalently, for every such $H$, we have $\mtt{F}_t(\mc{S}) \cap H\mc{S}^* \cap \mc{S}^* H \neq \emptyset$. 
Therefore, we have $\mtt{F}_t(\mc{S}) \cap H\mc{S} \cap \mc{S} H \neq \emptyset$ for every $H\in G(\N)\setminus \mc{S}$ such that $H^2 \neq C$. 
Now, on one hand, if there is no $H$ in $G(\N)\setminus \mc{S}$ such that $H^2 = C$, then we readily see that $\mc{S}$ is a symmetric unipotent numerical monoid. 
On the other hand, if there is such an $H$ in $G(\N)\setminus \mc{S}$, then we readily see that $\mc{S}$ together with $B:=H$ satisfy the conditions for $\mc{S}$ being a pseudo-symmetric unipotent numerical monoid. 
Hence, we finished the proof of our theorem.  
\end{proof}

We are now ready to prove our Theorem~\ref{intro:T7} which states the following: 
\medskip

Let $\mc{S}$ be a unipotent numerical monoid. 
If $\mc{S}$ is irreducible, then $\mc{S}$ is either symmetric or pseudo-symmetric.

\begin{proof}[Proof of Theorem~\ref{intro:T7}]
It is easy to check that every strongly symmetric unipotent numerical monoid is a symmetric unipotent numerical monoid. 
Similarly, every strongly pseudo-symmetric unipotent numerical monoid is a pseudo-symmetric unipotent numerical monoid. 
The rest of the proof follows from Theorem~\ref{intro:T7'}.
\end{proof}

The strong symmetry property imposes severe restrictions on the left and the right Frobenius sets.

\begin{Corollary}
If $\mc{S}$ is a strongly symmetric unipotent numerical monoid, then we have
\begin{align*}
\mtt{F}_r(\mc{S}) = \mtt{F}_l(\mc{S}) = \mtt{F}_t(\mc{S}).
\end{align*}
\end{Corollary}

\begin{proof}
We already know the inclusions $\mtt{F}_t(\mc{S}) \subseteq \mtt{F}_l(\mc{S})$ and $\mtt{F}_t(\mc{S}) \subseteq \mtt{F}_r(\mc{S})$. 
By Proposition~\ref{P:PF=F}, we know that the two-sided Frobenius set has a unique element, $\mtt{F}_t(\mc{S})=\{C\}$.
Assume towards a contradiction that there is an element $A\in  \mtt{F}_l(\mc{S}) \setminus \mtt{F}_t(\mc{S})$.
Since $A$ is from $G(\N)\setminus \mc{S}$, we know that $\mtt{F}_t(\mc{S}) \cap A \mc{S}^* \cap \mc{S}^* A \neq\emptyset$. 
Hence, there exists $B\in \mc{S}^*$ such that $C=AB$. 
But, as we pointed in Remark~\ref{R:absurd}, if $A$ and $C$ are two elements from $\mtt{F}_l(\mc{S})$, then there is no element $B\in G(\N)^*$ such that $AB=C$.
This contradiction shows that $ \mtt{F}_l(\mc{S}) = \mtt{F}_t(\mc{S})$. 
The equality $\mtt{F}_r(\mc{S}) = \mtt{F}_t(\mc{S})$ is proven similarly. 
Hence, our proof is complete.
\end{proof}

Recall our notation $\preceq$ from the introduction. It denotes the partial order $\leq_{G(\N),t}$ on $G(\N)$. 
We are now ready to prove the Theorem~\ref{intro:T8}.
We recall its statement for the convenience of the reader. 
\medskip

Let $\mc{S}$ be a strongly symmetric unipotent numerical monoid in $G(\N)$ such that $\mtt{F}_t(\mc{S}) = \{C\}$ for some $C\in G(\N)\setminus \mc{S}$. 
Then we have 
\begin{align*}
|\mtt{n}(\mc{S},C, \preceq) | = |\mtt{g}(\mc{S},C, \preceq)|= \mtt{g}(\mc{S}).
\end{align*}

\begin{proof}[Proof of Theorem~\ref{intro:T8}]

By Propositions~\ref{P:PF=F} and~\ref{P:PFisinF}, we know that $\{C\} = \mtt{F}_t(\mc{S})=\mtt{PF}_t(\mc{S})$.
Let $X\in G(\N)\backslash\mc{S}$. 
Then we have $\mtt{F}_t(\mc{S})\cap X \mc{S}^* \cap \mc{S}^* X \neq \emptyset$, implying that $X\preceq C$. 
This means that $G(\N)\setminus \mc{S}$ is a subset of $\mtt{g}(\mc{S},C,\preceq)$. 
But by definition, $\mtt{g}(\mc{S},C,\preceq)$ is a subset of $G(\N)\setminus \mc{S}$. 
Therefore we proved the following equalities of sets for symmetric numerical monoids:
\begin{align*}
G(\N)\setminus \mc{S} =\mtt{g}(\mc{S},C,\preceq) = \{ X\in G(\N)\setminus \mc{S}\mid \{ CX^{-1},X^{-1}C\}\subset \mc{S} \}.
\end{align*}
It follows that $\mtt{g}(\mc{S}) = | \mtt{g}(\mc{S},C,\preceq)|$.

Now, let us analyze the set $\mtt{n}(\mc{S},C, \preceq)$. 
This is precisely the set of all elements $A\in \mc{S}$ such that $\{A^{-1}C, CA^{-1}\} \subseteq G(\N)$. 
In addition, since $\mtt{n}(\mc{S},C, \preceq)\subset \mc{S}$, we must have that $\{ A^{-1}C , CA^{-1} \} \cap \mc{S} = \emptyset$. 
Otherwise, we would have $C\in \mc{S}$ which is absurd. 
In conclusion, the set $\mtt{n}(\mc{S},C, \preceq)$ is precisely the set of all elements $A\in \mc{S}$ such that 
$CA^{-1}\in G(\N)\setminus \mc{S}$ and $A^{-1}C\in G(\N)\setminus \mc{S}$.
We define a map $\Psi:\mtt{g}(\mc{S},C,\preceq) \to \mtt{n}(\mc{S},C,\preceq)$ as follows:
\begin{align*}
\Psi (X) := CX^{-1}\qquad (X\in \mtt{g}(\mc{S},C,\preceq)).
\end{align*}
It is evident from the fact that $\mc{S}$ is a cancellative monoid that $\Psi$ is an injective map.
We claim that this map is onto. 
Let $Y\in \mtt{n}(\mc{S},C,\preceq)$. We set $X:=Y^{-1}C$. 
Then we know from the previous paragraph that $X\in G(\N)\setminus \mc{S}$.
Hence, we see that $\Psi(X) = CX^{-1} = C (Y^{-1}C)^{-1} = Y$. 
This finishes the proof of the fact that the sets $\mtt{g}(\mc{S},C,\preceq)$ and $\mtt{n}(\mc{S},C,\preceq)$ are in a bijection with each other. 
\end{proof}

The last proposition of our paper shows how our previous theorem generalizes a well-known property of the symmetric generalized numerical semigroups.

\begin{Proposition}
Let $\mc{S}$ be a strongly symmetric unipotent numerical monoid. 
Then we have 
\begin{align*}
| \mtt{n}(G(\N),C,\leq_e) |= \mtt{g}(\mc{S}) + |\mtt{n}(S,C,\leq_e ) | 
\end{align*}
In particular, if $G$ is the commutative unipotent numerical monoid $\mathbf{P}(n,\N)$, then we have 
\begin{align*}
\text{vol}(\text{Cube}(C))= 2\mtt{g}(\mc{S}),
\end{align*}
where $\text{vol}(\text{Cube}(C))$ stands for the volume of the convex hull of points $P\in \N^{n}$, where $P$ is obtained from the vector $(1,c_2,\dots, c_n)$ by setting some of the coordinates to 0. Here, $(1,c_2,\dots, c_n)$ is the first row of the matrix $C$. 
\end{Proposition}

\begin{proof}
Let $\mtt{F}_t(\mc{S})$ be given by $\{C\}$.
By Propositions~\ref{P:PF=F} and~\ref{P:PFisinF}, we know that 
\begin{align*}
\{C\} = \mtt{F}_t(\mc{S})=\mtt{PF}_t(\mc{S}).
\end{align*} 
This means that for every gap set element $X\in G(\N)\setminus \mc{S}$, we have $X\leq_e C$.
In particular, we see that 
\begin{align}\label{A:leqcleqt}
\{ X\in G(\N)\setminus \mc{S} \mid X\leq_e C \} = \{ X\in G(\N) \setminus \mc{S} \mid X \leq_{\mc{S},t} C\}.  
\end{align}
Since $C$ is the unique maximal element of $G(\N)\setminus \mc{S}$ with respect to $\leq_{\mc{S},t}$,
(\ref{A:leqcleqt}) implies that 
\begin{align}\label{A:triviallyintersect}
G(\N)\setminus \mc{S} = \mtt{n}(G(\N)\setminus \mc{S} , C ,\leq_e).
\end{align} 
Then we see that 
\begin{align*}
| \mtt{n}(G(\N),C,\leq_e) | &= | \{ B\in G(\N)\setminus\mc{S} \mid  B\leq_e C \} | + | \{ B\in \mc{S} \mid B\leq_e C\}|  \\
&= | G(\N)\setminus\mc{S} | + | \mtt{n}(S,C,\leq_e )|.
\end{align*}
This finishes the proof of our first assertion.

We now proceed with the assumption that $G= \mathbf{P}(n,\N)$.
Then we notice that the orders $\leq_e$ and $\preceq$ on a unipotent numerical monoid $\mc{S}$ in $G$ are the same. 
Indeed, an element $B\in \mc{S}$ covers another element $A\in \mc{S}$ in $\leq_e$ if and only if the difference $B-A$ has all nonnegative entries.
By using elementary matrices, we can express this as follows: 
\begin{align*}
B=AE_{1,j} \quad \text{ for some $1\leq j\leq n$},
\end{align*}
where $E_{1,j}$ is the elementary unipotent matrix with 1 at its $(1,j)$-th entry. 
Hence, by Theorem~\ref{intro:T8}, we see that $| \mtt{n}(G(\N),C,\leq_e) | = 2\mtt{g}(S)$.
Let $(1,c_2,\dots, c_n)$ denote the first row of $C$. 
Then $|\mtt{n}(G(\N),C,\leq_e)|$ is given by the number of points in the hypercube whose vertices are obtained from $(1,c_2,\dots, c_n)$ by setting some of the coordinates to 0. 
But this proves our second assertion.
Hence, the proof of our proposition is complete. 
\end{proof}

\section*{Acknowledgements}

The authors are grateful to the referee. 
The comments and remarks of the referee not only significantly improved the quality and the strength of this article but also clarified some of its obscurities. 
The first author gratefully acknowledges partial support from the Louisiana Board of Regents grant, contract no. LEQSF(2023-25)-RD-A-21.
	
\bibliography{references.bib}

\begin{thebibliography}{10}

\bibitem{ADG}
Abdallah Assi, Marco D'Anna, and Pedro~A. Garc\'{\i}a-S\'{a}nchez.
\newblock {\em Numerical semigroups and applications}, volume~3 of {\em RSME
  Springer Series}.
\newblock Springer, Cham, [2020] \copyright 2020.
\newblock Second edition [of 3558713].

\bibitem{BeelenTutas}
Peter Beelen and Nesrin Tuta\c{s}.
\newblock A generalization of the {W}eierstrass semigroup.
\newblock {\em J. Pure Appl. Algebra}, 207(2):243--260, 2006.

\bibitem{BTT2022}
Matheus Bernardini, Wanderson Ten\'{o}rio, and Guilherme Tizziotti.
\newblock The corner element of generalized numerical semigroups.
\newblock {\em Results Math.}, 77(4):Paper No. 141, 20, 2022.

\bibitem{BGS2023}
Om~Prakash Bhardwaj, Kriti Goel, and Indranath Sengupta.
\newblock Affine semigroups of maximal projective dimension.
\newblock {\em Collect. Math.}, 74(3):703--727, 2023.

\bibitem{CanSakran}
Mahir~Bilen Can and Naufil Sakran.
\newblock On generalized {W}ilf conjectures.
\newblock {\em Port. Math.}, 81(1-2):21--55, 2024.

\bibitem{CarvalhoKato}
C\'{\i}cero Carvalho and Takao Kato.
\newblock On {W}eierstrass semigroups and sets: a review with new results.
\newblock {\em Geom. Dedicata}, 139:195--210, 2009.

\bibitem{CDGS}
Carmelo Cisto, Manuel Delgado, and Pedro~A. Garc\'{\i}a-S\'{a}nchez.
\newblock Algorithms for generalized numerical semigroups.
\newblock {\em J. Algebra Appl.}, 20(5):Paper No. 2150079, 24, 2021.

\bibitem{CFPU2019}
Carmelo Cisto, Gioia Failla, Chris Peterson, and Rosanna Utano.
\newblock Irreducible generalized numerical semigroups and uniqueness of the
  {F}robenius element.
\newblock {\em Semigroup Forum}, 99(2):481--495, 2019.

\bibitem{CFU2019}
Carmelo Cisto, Gioia Failla, and Rosanna Utano.
\newblock On the generators of a generalized numerical semigroup.
\newblock {\em An. \c{S}tiin\c{t}. Univ. ``Ovidius'' Constan\c{t}a Ser. Mat.},
  27(1):49--59, 2019.

\bibitem{CistoTenorio2021}
Carmelo Cisto and Wanderson Ten\'{o}rio.
\newblock On almost-symmetry in generalized numerical semigroups.
\newblock {\em Comm. Algebra}, 49(6):2337--2355, 2021.

\bibitem{DGSR}
M.~Delgado, P.~A. Garc\'{\i}a-S\'{a}nchez, and J.~C. Rosales.
\newblock Numerical semigroups problem list.
\newblock {\em {B}ulletin of {C}entro {I}nternacional de {M}atemetica},
  7:8--14, 2013.

\bibitem{DGMV}
J.~D. D\'{\i}az-Ram\'{\i}rez, J.~I. Garc\'{\i}a-Garc\'{\i}a,
  D.~Mar\'{\i}n-Arag\'{o}n, and A.~Vigneron-Tenorio.
\newblock Characterizing affine {$\mathcal {C}$}-semigroups.
\newblock {\em Ric. Mat.}, 71(1):283--296, 2022.

\bibitem{DiPasqualeGillespieBryanPeterson}
Michael DiPasquale, Bryan~R. Gillespie, and Chris Peterson.
\newblock Quasi-polynomial growth of numerical and affine semigroups with
  constrained gaps.
\newblock {\em Semigroup Forum}, 107(1):60--78, 2023.

\bibitem{FPU2016}
Gioia Failla, Chris Peterson, and Rosanna Utano.
\newblock Algorithms and basic asymptotics for generalized numerical semigroups
  in {$\Bbb{N}^d$}.
\newblock {\em Semigroup Forum}, 92(2):460--473, 2016.

\bibitem{GMV2018}
J.~I. Garc\'{\i}a-Garc\'{\i}a, D.~Mar\'{\i}n-Arag\'{o}n, and
  A.~Vigneron-Tenorio.
\newblock An extension of {W}ilf's conjecture to affine semigroups.
\newblock {\em Semigroup Forum}, 96(2):396--408, 2018.

\bibitem{Homma1996}
Masaaki Homma.
\newblock The {W}eierstrass semigroup of a pair of points on a curve.
\newblock {\em Arch. Math. (Basel)}, 67(4):337--348, 1996.

\bibitem{Kim1994}
Seon~Jeong Kim.
\newblock On the index of the {W}eierstrass semigroup of a pair of points on a
  curve.
\newblock {\em Arch. Math. (Basel)}, 62(1):73--82, 1994.

\bibitem{Matthews2001}
Gretchen~L. Matthews.
\newblock Weierstrass pairs and minimum distance of {G}oppa codes.
\newblock {\em Des. Codes Cryptogr.}, 22(2):107--121, 2001.

\bibitem{RosalesGarciaSanchez}
J.~C. Rosales and P.~A. Garc\'{\i}a-S\'{a}nchez.
\newblock {\em Finitely generated commutative monoids}.
\newblock Nova Science Publishers, Inc., Commack, NY, 1999.

\bibitem{RGS2}
J.~C. Rosales and P.~A. Garc\'{\i}a-S\'{a}nchez.
\newblock {\em Numerical semigroups}, volume~20 of {\em Developments in
  Mathematics}.
\newblock Springer, New York, 2009.

\bibitem{SinghaiLin2022}
Deepesh Singhal and Yuxin Lin.
\newblock Frobenius allowable gaps of generalized numerical semigroups.
\newblock {\em Electron. J. Combin.}, 29(4):Paper No. 4.12, 21, 2022.

\bibitem{TT2019}
Wanderson Ten\'{o}rio and Guilherme Tizziotti.
\newblock On {W}eierstrass gaps at several points.
\newblock {\em Bull. Braz. Math. Soc. (N.S.)}, 50(2):543--559, 2019.

\end{thebibliography}
\bibliographystyle{plain}
\end{document}